\documentclass{amsart}%
\usepackage{amsfonts}
\usepackage{amsmath}
\usepackage{amssymb}
\usepackage{graphicx}%
\newtheorem{theorem}{Theorem}
\theoremstyle{plain}

\newtheorem{corollary}{Corollary}

\newtheorem{definition}{Definition}

\newtheorem{lemma}{Lemma}

\newtheorem{proposition}{Proposition}
\newtheorem{remark}{Remark}

\numberwithin{equation}{section}
\begin{document}
\title{MOTIVIC ZETA FUNCTIONS FOR CURVE SINGULARITIES}
\author{J. J. Moyano-Fern\'{a}ndez}
\address{Institut f\"ur Mathematik, Universit\"at Osnabr\"uck. Albrechtstrasse 28a,
49076 Osnabr\"uck, Deutschland}
\email{jmoyano@mathematik.uni-osnabrueck.de}
\author{W. A. Z\'{u}\~{n}iga-Galindo}
\address{Centro de Investigaci\'{o}n y de Estudios Avanzados del I.P.N., Departamento
de Mate\-m\'{a}ticas, Av. Instituto Polit\'{e}cnico Nacional 2508, Col. San
Pedro Zacatenco, M\'{e}xico D.F., C.P. 07360, M\'{e}xico }
\email{wzuniga@math.cinvestav.mx}
\thanks{The first author was partially supported by the grant MEC MTM2007-64704, by
Junta de CyL VA065A07, and by the grant DAAD-La Caixa.}

\begin{abstract}
Let $X$ be a complete, geometrically irreducible, singular, algebraic curve
defined over a field of characteristic $p$ big enough. Given a local ring
$\mathcal{O}_{P,X}$ at a rational singular point $P$ of $X$, we attached a
universal zeta function which is a rational function and admits a functional
equation if $\mathcal{O}_{P,X}$\ is Gorenstein. This universal zeta function
specializes to other known zeta functions and Poincar\'{e} series attached to
singular points of algebraic curves. In particular, for the local ring
attached to a complex analytic function in two variables, our universal zeta
function specializes to the generalized Poincar\'{e} series introduced by
Campillo, Delgado and Gusein-Zade.

\end{abstract}
\maketitle

\section{Introduction}

Let $X$ be a complete, geometrically irreducible, singular, algebraic curve
defined over a finite field $\mathbb{F}_{q}$. In \cite{Z1} the second author
introduced a zeta function \ $Z($\textrm{Ca}$(X),T)$ associated to the
effective Cartier divisors on $X$. Other types of zeta functions associated to
singular curves over finite fields were introduced in \cite{Ga}, \cite{Green},
\cite{Stohr}, \cite{stohr}, \cite{Z3}. The zeta function $Z($\textrm{Ca}%
$(X),T)$\ admits an Euler product with non-trivial factors at the singular
points of $X$. If $P$ is a rational singular point of $X$, then the local
factor $Z_{\mathrm{Ca}(X)}(T,q,\mathcal{O}_{P,X})$ \ at $P$ is a rational
function of $T$ depending on $q$ and the completion $\widehat{\mathcal{O}%
}_{P,X}$ of the local ring $\mathcal{O}_{P,X}$ of $X$ at $P$. If the residue
field of $\widehat{\mathcal{O}}_{P,X}$ is not too small, then $Z_{\mathrm{Ca}%
(X)}(T,q,\mathcal{O}_{P,X})$ depends only on the semigroup of $\widehat
{\mathcal{O}}_{P,X}$ (see \cite[Theorem 5.5]{Z1}). Thus, if $\widehat
{\mathcal{O}}_{P,X}\cong\mathbb{F}_{q} [\![ x,y ] \!] /\left(  f(x,y)\right)
$, then $Z_{\mathrm{Ca}(X)}(T,q,\mathcal{O}_{P,X})$ becomes a complete
invariant of the equisingularity class of the algebroid curve $\widehat
{\mathcal{O}}_{P,X}$ (see \cite{CDG}, \cite{W}, \cite{ZA2}). Motivated by
\cite{DL0}, in \cite{Z2} the second author computed several examples showing
that $\lim_{q\rightarrow1}Z_{\mathrm{Ca}(X)}(T,q,\mathcal{O}_{P,X})$ equals
the zeta function of the monodromy of the (complexification) of $f$ at the
origin (see \cite{A}, and the examples in Section \ref{ejemplos}). This paper
aims to study this phenomenon. \medskip

By using motivic integration in the spirit of Campillo, Delgado and
Gusein-Zade we attach to a local ring $\mathcal{O}_{P,X}$ of an algebraic
curve $X$ a `universal zeta function' (see Definition \ref{def5}, Theorem
\ref{teo1}, De\-finition \ref{universalzetafunction}). This zeta function
specializes to $Z_{\mathrm{Ca}(X)}(T,q,\mathcal{O}_{P,X})$ (see Lemma
\ref{TheoremB} and Theorem \ref{TheoremD}). We also establish that
$\lim_{q\rightarrow1}Z_{\mathrm{Ca}(X)}(T,q,\mathcal{O}_{P,X})$ equals to a
zeta function of the monodromy of a reduced complex mapping in two variables
at the origin (see Theorem \ref{TheoremD}). A key ingredient is a result of
Campillo, Delgado and Gusein-Zade relating the Poincar\'{e} series attached to
complex analytic functions in two variables and the zeta function of the
monodromy (see \cite{CDG}, and Theorem \ref{TheoremC}). From the point of view
of the work of Campillo, Delgado and Gusein-Zade, this paper deals with
Poincar\'{e} series attached to local rings $\mathcal{O}_{P,X}$ when the
ground field is big enough (see Lemma \ref{rem5}). In particular, for the
local ring attached to a complex analytic function in two variables, our
universal zeta function specializes to the generalized Poincar\'{e} series
introduced in \cite{CDG4}, and then a relation with the Alexander polynomial
holds as a consequence of \cite{CDG1}. We also obtain explicit formulas that
give precise information about the degree of the numerators of such
Poincar\'{e} series and functional equations (see Theorem \ref{TheoremE} and
the corollaries following it). Our results suggest that the factor
$Z_{\mathrm{Ca}(X)}(T,q,\mathcal{O}_{P,X})$ is the `monodromy zeta function of
$\mathcal{O}_{P,X}$'. In order to understand this, we believe that a
cohomological theory for the universal zeta functions should be developed.

Finally, we want to comment that the connections between zeta functions
introduced here and the motivic zeta functions of Kapranov \cite{Ka} and
Baldassarri-Deninger-Naumann \cite{B-D-N} are unknown. However, we believe
that the zeta functions introduced here are factors of motivic zeta functions
of Baldassarri-Deninger-Nau\-mann type for singular curves. In a forthcoming
paper the authors plan to study this connection. For a general discussion
about motivic zeta functions for curves the reader may consult \cite[and the
references therein]{An} and \cite{DL1}.

\textbf{Acknowledgement.} The authors wish to thank the referee for his or her
useful comments, which led to an improvement of this work.

\section{\label{semigroup}The Semigroup of Values of a Curve Singularity}

Let $X$ be a complete, geometrically irreducible, algebraic curve defined over
a field $k$, with function field $K/k$. Let $\widetilde{X}$ be the
normalization of $X$ over $k$ and let $\pi:\widetilde{X}\rightarrow X$ be the
normalization map. Let $P\in X$ be a closed point of $X$ and $O_{P}=O_{P,X}$
the local ring of $X$ at $P$. Let $Q_{1},\ldots,Q_{d}$ be the points of
$\widetilde{X}$ lying over $P$, i.e., $\pi^{-1}\left(  P\right)  =\left\{
Q_{1},\ldots,Q_{d}\right\}  $, and let $O_{Q_{1}},\ldots,O_{Q_{d}}$\ be the
corresponding local rings at these points. Since the function fields of
$\widetilde{X}$ and $X$ are the same, and $\widetilde{X}$ is a non-singular
curve, the local rings $O_{Q_{1}},\ldots,O_{Q_{d}}$\ are valuation rings of
$K/k$ over $O_{P}$. The integral closure of $O_{P}$ in $K/k$ is $\mathbb{O}%
_{P}=$ $O_{Q_{1}}\cap\ldots\cap O_{Q_{d}}$.

Let $\widehat{\mathbb{O}}_{P}$ be the completion of $\mathbb{O}_{P}$ with
respect to its Jacobson ideal, and let $\widehat{O}_{P}$ be, respectively
$\widehat{O}_{Q_{i}}$ for $i=1,\ldots,d$, the completion of $O_{P}$,
respectively of $O_{Q_{i}}$ for $i=1,\ldots,d$, with respect to the topology
induced by their maximal ideals. We denote by $B_{P}^{(j)}$, $j$ $=1,\ldots
,d$, the minimal primes of $\widehat{O}_{P}$. Then we have the following diagram:%

\[%
\begin{array}
[c]{lll}%
\widehat{\mathbb{O}}_{P} & \underrightarrow{\cong} & \widehat{O}_{Q_{1}}%
\times\ldots\times\widehat{O}_{Q_{d}}\\
\uparrow &  & \uparrow\\
\widehat{O}_{P} & \underrightarrow{\underline{\varphi}} & \widehat{O}%
_{B_{P}^{(1)}}\times\ldots\times\widehat{O}_{B_{P}^{(d)}},
\end{array}
\]
where $\underline{\varphi}$ is the diagonal morphism. Since $\widehat{O}_{P}$
is a reduced ring (cf. \cite[Theorem 1]{Ros1}) and \cite[proof of Satz
3.6]{H-K}), $\underline{\varphi}$ is one to one. Thus we have a bijective
correspondence \ between the $\widehat{O}_{Q_{i}}$'s and $\widehat{O}%
_{B_{P}^{(i)}}$'s. We call the rings $\widehat{O}_{B_{P}^{(i)}}$ \textit{the
branches} of $\widehat{O}_{P}$. By the Cohen structure theorem for complete
regular local rings, each $\widehat{O}_{Q_{i}}$ is isomorphic to $k_{i} [\![
t_{i} ]\!] $, $i=1,\ldots,d$, where $k_{i}$ is the residue field of
$\widehat{O}_{Q_{i}}$.

We will say that $\widehat{O}_{P}$ is\textit{ totally rational} if all rings
$\widehat{O}_{Q_{i}}$, for $i=1,\ldots,d$, have $k$ as residue field.

\textit{From now on we assume that }$\widehat{O}_{P}$\textit{ is totally
rational ring and identify }$\widehat{O}_{P}$\textit{ with }$\varphi\left(
\widehat{O}_{P}\right)  $. Let $v_{i}$ denote the valuation associated with
$\widehat{O}_{Q_{i}}$, $i=1,\ldots,d$. By using these valuations we define
$\underline{v}\left(  \underline{z}\right)  =\left(  v_{1}\left(
z_{1}\right)  ,\ldots,v_{d}\left(  z_{d}\right)  \right)  $, for any non-zero
divisor $\underline{z}=\left(  z_{1},\ldots,z_{d}\right)  \in\widehat
{\mathbb{O}}_{P}$.

\textit{The semigroup }$S$\ \textit{of values} \textit{of }$\widehat{O}_{P}$
consists of all the elements of the form $\underline{v}\left(  \underline
{z}\right)  =\left(  v_{1}\left(  z_{1}\right)  ,\ldots,v_{d}\left(  z_{d}
\right)  \right)  \in\mathbb{N}^{d}$ for all the non-zero divisors
$\underline{z}\in\widehat{O}_{P}$. Observe that, by definition, the semigroup
of $\widehat{O}_{P}$ coincides with the semigroup of values of $O_{P}$.

We set $\underline{z}=\underline{t}^{\underline{n}}\underline{\mu}:=\left(
t_{1}^{n_{1}},\ldots,t_{d}^{n_{d}}\right)  \left(  \mu_{1},\ldots,\mu
_{d}\right)  =\left(  t_{1}^{n_{1}}\mu_{1},\ldots,t_{d}^{n_{d}}\mu_{d}\right)
$, with $\underline{\mu}=\left(  \mu_{1},\ldots,\mu_{d}\right)  \in
\widehat{\mathbb{O}}_{P}^{\times}$. With this notation, the ideal generated by
a non-zero divisor of $\widehat{\mathbb{O}}_{P}$ has the form $\underline
{t}^{\underline{n}}\widehat{\mathbb{O}}_{P}$, for some $\underline{n}%
\in\mathbb{N}^{d}$.

We set $\underline{1}:=\left(  1,\ldots,1\right)  \in\mathbb{N}^{d}$ and, for
$\underline{n}=\left(  n_{1},\ldots,n_{d}\right)  \in\mathbb{N}^{d}$,
$\left\Vert \underline{n}\right\Vert :=n_{1}+\ldots+n_{d}$. We introduce a
partial order in $\mathbb{N}^{d}$, \textit{the product order}, by taking
$\underline{n}\geq\underline{m}$, if $n_{i}\geq m_{i}$ for $i=1,\ldots,d$.

There exists $\underline{c}_{P}=\left(  c_{1},\ldots,c_{d}\right)
\in\mathbb{N}^{d}$ minimal for the product order such that $\underline{c}%
_{P}+\mathbb{N}^{d}\subseteq S$. This element is called\textit{\ the
conductor} of $S$. The\textit{\ conductor ideal} $\widehat{F}_{P}$ of
$\widehat{O}_{P}$ is $\underline{t}^{\underline{c}_{P}}\widehat{\mathbb{O}%
}_{P}$. This is the largest common ideal of $\widehat{O}_{P}$ and
$\widehat{\mathbb{O}}_{P}$. The \textit{\ singularity degree } $\delta_{P}%
$\ of $\widehat{O}_{P}$ is defined as $\delta_{P}:=\dim_{k}\widehat
{\mathbb{O}}_{P}/\widehat{O}_{P}<\infty$ \ (see e.g. \cite[Chapter IV]{Ser}).
If $\widehat{O}_{P}$ is a Gorenstein ring, the singularity degree is related
to the conductor by the equality $\left\Vert \underline{c}_{P}\right\Vert
=2\delta_{P}$ (see e.g. \cite[Chapter IV]{Ser}). By using the fact that
$\widehat{O}_{P}/\widehat{F}_{P}$ is a $k$-subalgebra of $\widehat{\mathbb{O}%
}_{P}/\widehat{F}_{P}$ of codimension $\delta_{P}$, that $\widehat{\mathbb{O}%
}_{P}/\widehat{F}_{P}$ is a finite dimensional $k$-algebra, and \ that
$\widehat{F}_{P}$ is a common ideal of $\widehat{O}_{P}$ and $\widehat
{\mathbb{O}}_{P}$, we have%
\begin{equation}
\widehat{O}_{P}=\left\{  \left(
{\textstyle\sum\nolimits_{i=0}^{\infty}}
a_{i,1}t_{1}^{i},\ldots,%
{\textstyle\sum\nolimits_{i=0}^{\infty}}
a_{i,d}t_{d}^{i}\right)  \in\widehat{\mathbb{O}}_{P}\mid\Delta=0\right\}
\label{delta}%
\end{equation}
where $\Delta=0$ denotes a homogeneous system of linear equations involving
only a finite number of the $a_{i,j}$. Indeed,
\[
c_{m}=1+\max\left\{  i\mid a_{i,m}\text{ appears in }\Delta=0\right\}  ,
\]
for $m=1,\ldots,d$ (see examples in Section \ref{ejemplos}). Note that, as a
consequence of the definition of $\underline{\varphi}$,\ the relations
$a_{0,1}=a_{0,2}=\ldots=a_{0,d}$ hold.

\begin{remark}
[Conventions and Notation](1) From now on we will use `$X$ is an algebraic
curve over $k$', to mean that $X$ is a complete, geometrically irreducible,
algebraic curve over $k$.

\noindent(2) To simplify the notation, we drop the index $P$, and denote
$\widehat{O}_{P}$ by $\mathcal{O}$, $\widehat{F}_{P}$ by $\mathcal{F}$ and
$\widehat{\mathbb{O}}_{P}$ by $\widetilde{\mathcal{O}}=k [\![ t_{1} ]\!]
\times\ldots k [\![ t_{d} ]\!] $, and $\mathcal{O}$ is a $k$-vector space of
finite codimension in $\widetilde{\mathcal{O}}$ with presentation
(\ref{delta}). We also drop the index $P$ from $\underline{c}_{P}$ and
$\delta_{P}$.
\end{remark}

\begin{remark}
\label{rem-1} Let $\left(  X,0\right)  \subset\left(  \mathbb{C}^{2},0\right)
$ be a germ of reduced plane curve given by $f=0$ for $f\in\mathcal{O}%
_{\left(  \mathbb{C}^{2},0\right)  }$, and let $X=%
{\textstyle\bigcup\nolimits_{i=1}^{d}}
X_{i}$ with $d\geq1$ be its decomposition into irreducible components (or
branches) corresponding to $f=%
{\textstyle\prod\nolimits_{i=1}^{d}}
f_{i}$. Let $\mathcal{O}:=\mathcal{O}_{\left(  X,0\right)  }=\mathcal{O}%
_{\left(  \mathbb{C}^{2},0\right)  }/\left(  f\right)  $ be the ring of germs
of analytic functions on $X$. \ Let $\varphi_{i}:\left(  \mathbb{C},0\right)
\rightarrow\left(  \mathbb{C}^{2},0\right)  $ be a \ parametrization of
$X_{i}$, i.e., $\varphi_{i}$\ is an isomorphism between $X_{i}$ and
$\mathbb{C}$ outside of the origin, for $i=1,\ldots,d$. Let $S(\mathcal{O}%
):=S(f)$ denote the semigroup of $\mathcal{O}$ defined by using the
parametrizations $\varphi_{i}$'s. (For further details, see e.g. \cite{Del}).
\end{remark}

\section{\label{Euler}Integration with Respect to the Generalized Euler
Characteristic}

We denote by $Var_{k}$ the category of $k$-algebraic varieties, and by
$K_{0}\left(  Var_{k}\right)  $\ the corresponding Grothendieck ring. It is
the ring generated by symbols $\left[  V\right]  $, for $V$ an algebraic
variety, with the relations $\left[  V\right]  =\left[  W\right]  $ if $V$ is
isomorphic to $W$, $\left[  V\right]  =\left[  V\setminus Z\right]  +\left[
Z\right]  $ if $Z$ is closed in $V$, and $\left[  V\times W\right]  =\left[
V\right]  \left[  W\right]  $. We denote $\boldsymbol{1}:=\left[
\text{point}\right]  $, $\mathbb{L}:=\left[  \mathbb{A}_{k}^{1}\right]  $ and
$\mathcal{M}_{k}:=K_{0}\left(  Var_{k}\right)  $\ $\left[  \mathbb{L}%
^{-1}\right]  $ the ring obtained by localization with respect to the
multiplicative set generated by $\mathbb{L}$.

We define \textit{the set of }$\underline{n}$\textit{-jets} $J_{\widetilde
{\mathcal{O}}}^{\underline{n}}$ \textit{of the local ring }$\widetilde
{\mathcal{O}}$\textit{ }as $J_{\widetilde{\mathcal{O}}}^{\underline{n}%
}=\widetilde{\mathcal{O}}/t^{\underline{n}+\underline{1}}\widetilde
{\mathcal{O}}\cong k^{\left\Vert \underline{n}+\underline{1}\right\Vert }$.
The canonical projection $\widetilde{\mathcal{O}}$ $\rightarrow$
$\widetilde{\mathcal{O}}/\underline{t}^{\underline{n}+\underline{1}}%
\widetilde{\mathcal{O}}$ is denoted by $\pi_{\underline{n}}$.

\begin{definition}
\label{def1}A subset $X\subseteq\widetilde{\mathcal{O}}=k [\![ t_{1} ]\!]
\times\ldots\times k [\![ t_{d} ]\!]$ is said to be cylindric if
$X=\pi_{\underline{n}}^{-1}\left(  Y\right)  $ for a constructible subset $Y$
of $J_{\widetilde{\mathcal{O}}}^{\underline{n}}$.
\end{definition}

We note that $\mathcal{O}$ and $\mathcal{O}^{\times}$ (the group of units of
$\mathcal{O}$)\ are cylindric subsets of $\widetilde{\mathcal{O}}$ (cf.
(\ref{delta})).

\begin{remark}
\label{rem0} Any constructible subset $Y$ of $J_{\widetilde{\mathcal{O}}%
}^{\underline{n}}$ is defined by a condition that can be expressed as a finite
Boolean combination of conditions of the form%
\[
\left\{
\begin{array}
[c]{ll}%
p_{i}\left(  x_{0},\ldots,x_{m-1}\right)  =0, & i\in I;\\
& \\
q\left(  x_{0},\ldots,x_{m-1}\right)  \neq0, &
\end{array}
\right.
\]
where $m=\left\Vert \underline{n}+\underline{1}\right\Vert $, the
$p_{i}\left(  x_{0},\ldots,x_{m-1}\right)  $, $q\left(  x_{0},\ldots
,x_{m-1}\right)  $ are polynomials in $k\left[  x_{0},\ldots,x_{m-1}\right]
$, and $I$ is a finite subset independent of $m$. We call such a condition
\textit{constructible in }$J_{\widetilde{\mathcal{O}}}^{\underline{n}}$.
Definition \ref{def1} means that the condition for a function $\underline
{z}\in\widetilde{\mathcal{O}}$ to belong to the set $X$ is a constructible
condition on the $\underline{n}$\textit{-jet }$\pi_{\underline{n}}\left(
\underline{z}\right)  $\textit{\ of }$\underline{z}$.
\end{remark}

We present now the notion of integral with respect to the generalized Euler
characteristic introduced by Campillo, Delgado and Gusein-Zade in \cite{CDG4}
for the complex case (and in \cite{DM} for more general contexts).

\begin{definition}
\label{def2} The generalized Euler characteristic (or motivic measure) of a
cylindric subset $X\subseteq\widetilde{\mathcal{O}}$, $X=\pi_{\underline{n}%
}^{-1}\left(  Y\right)  $, with $Y\subseteq J_{\widetilde{\mathcal{O}}%
}^{\underline{n}}$ constructible, is $\chi_{g}\left(  X\right)  :=\left[
Y\right]  \mathbb{L}^{-\left\Vert \underline{n}+\underline{1}\right\Vert }%
\in\mathcal{M}_{k}$.
\end{definition}

The generalized Euler characteristic $\chi_{g}\left(  X\right)  $ is well
defined since, if $X=\pi_{\underline{m}}^{-1}\left(  Y^{\prime}\right)  $,
$Y^{\prime}\subseteq J_{\widetilde{\widehat{\mathcal{O}}}}^{\underline{m}}$,
$\underline{n}\geqslant\underline{m}$, then $Y$ is a locally trivial fibration
over $Y^{\prime}$ with fiber $k^{r}$, where $r=\left\Vert \underline
{n}+\underline{1}\right\Vert -\left\Vert \underline{m}+\underline
{1}\right\Vert $.

\begin{definition}
\label{def3}Let $\left(  G,+,0\right)  $ be an Abelian group, and $X$ a
cylindric subset of $\widetilde{\mathcal{O}}$. A function $\phi:\widetilde
{\mathcal{O}}\rightarrow G$ is called \textit{cylindric} if it has countably
many values and, for each $a\in G$, $a\neq0$, the set $\phi^{-1}\left(
a\right)  $ is cylindric. As in \ \cite{DL4}, \cite{CDG4} we define
\[%
{\textstyle\int\limits_{X}}
\phi d\chi_{g}=%
{\textstyle\sum\limits_{\substack{a\in G\\a\neq0}}}
\chi_{g}\left(  X\cap\phi^{-1}\left(  a\right)  \right)  \otimes a,
\]
if the sum has sense in $G\otimes_{\mathbb{Z}}\mathcal{M}_{k}$. In such a case
the function $\phi$\textit{\ is said to be integrable over }$X$.
\end{definition}

Now we give the projective versions of the above definitions which we will use
later on. For a $k$-vector space $L$ (finite or infinite dimensional), let
$\mathbb{P}L=\left(  L\setminus\left\{  0\right\}  \right)  /k^{\times}$ be
its projectivization, let $\mathbb{P}^{\times}L$ be the disjoint union of
$\mathbb{P}L$ with a point ($\mathbb{P}^{\times}L$ can be identified with
$L/k^{\times}$). The natural map $\mathbb{P}\widetilde{\mathcal{O}}%
\rightarrow\mathbb{P}^{\times}J_{\widetilde{\mathcal{O}}}^{\underline{n}}$ is
also denoted by $\pi_{\underline{n}}$.

\begin{definition}
\label{def4}A subset $X\subseteq\mathbb{P}\widetilde{\mathcal{O}}$ is said to
be cylindric if $X=\pi_{\underline{n}}^{-1}\left(  Y\right)  $ for a
constructible subset $Y$ of $\mathbb{P}J_{\widetilde{\mathcal{O}}}%
^{\underline{n}}\subset\mathbb{P}^{\times}J_{\widetilde{\mathcal{O}}%
}^{\underline{n}}$. The generalized Euler characteristic $\chi_{g}\left(
X\right)  $ of $X$ is $\chi_{g}\left(  X\right)  :=\left[  Y\right]
\mathbb{L}^{-\left\Vert \underline{n}+\underline{1}\right\Vert }\in
\mathcal{M}_{k}$.
\end{definition}

A function $\phi:\mathbb{P}\widetilde{\mathcal{O}}\rightarrow G$ is called
\textit{cylindric }if it satisfies the conditions in Defi\-nition \ref{def3}.
The notion of integration over a cylindric subset of $\mathbb{P}%
\widetilde{\mathcal{O}}$ with respect $d\chi_{g}$ follows the pattern of
Definition \ref{def3}.

\begin{remark}
\label{rem1}Let $V$ be a cylindric subset and a $k$-vector subspace of
$\widetilde{\mathcal{O}}$. Let $\pi$ be the factorization map $\widetilde
{\mathcal{O}}\setminus\left\{  0\right\}  \rightarrow\mathbb{P}\widetilde
{\mathcal{O}}$, $\Omega:\mathbb{P}\widetilde{\mathcal{O}}\rightarrow G$ a
cylindric function integrable over $\mathbb{P}V$, and define $\overline
{\Omega}:=\Omega\circ\pi:\widetilde{\mathcal{O}}\setminus\left\{  0\right\}
\rightarrow G$. Then $\overline{\Omega}$ is cylindric function integrable over
$V$ and
\begin{equation}%
{\textstyle\int\limits_{V}}
\overline{\Omega}d\chi_{g}=\left(  \mathbb{L}-1\right)
{\textstyle\int\limits_{\mathbb{P}V}}
\Omega d\chi_{g}. \label{G5}%
\end{equation}
The identity follows from the fact that
\[
\chi_{g}\left(  \overline{\Omega}^{-1}\left(  a\right)  \cap V\right)
=\left(  \mathbb{L}-1\right)  \chi_{g}\left(  \Omega^{-1}\left(  a\right)
\cap\mathbb{P}V\right)  \text{, for }a\in G,a\neq0\text{.}%
\]

\end{remark}

\section{\label{groupJ}The Structure of the Algebraic Group $\mathcal{J}$}

In this section $k$ is a field of characteristic zero. The quotient group
$\widetilde{\mathcal{O}}^{\times}/\left(  \underline{1}+\mathcal{F}\right)  $
admits a polynomial system of representatives $\left(  g_{1},\ldots
,g_{i},\ldots,g_{d}\right)  $, where \ $g_{i}=%
{\textstyle\sum\nolimits_{j=0}^{c_{i}-1}}
a_{j,i}t_{i}^{j}$, with $a_{0,i}\in k^{\times}$ and $\underline{c}=\left(
c_{1},\ldots,c_{d}\right)  $ is the conductor of $S$. Thus $\widetilde
{\mathcal{O}}^{\times}/\left(  \underline{1}+\mathcal{F}\right)  $ can be
considered as an open subset of the affine space of dimension $\left\Vert
\underline{c}\right\Vert $, this algebraic structure is compatible with the
group structure of $\widetilde{\mathcal{O}}^{\times}/\left( \underline
{1}+\mathcal{F} \right) $ (cf. \cite[Chapter V, Section 14]{Ser}).
Furthermore,
\[
\widetilde{\mathcal{O}}^{\times}/\left(  \underline{1}+\mathcal{F}\right)
\cong\left(  G_{m}\right)  ^{d}\times\left(  G_{a}\right)  ^{\left\Vert
\underline{c}\right\Vert -d}\text{,}%
\]
as algebraic groups, where $G_{m}=\left(  k^{\times},\cdot\right)  $,
$G_{a}=\left(  k,+\right)  $, (cf. \cite[Chapter V, Section 14]{Ser}). By the
previous discussion, the group $\mathcal{O}^{\times}/\left(  \underline
{1}+\mathcal{F}\right)  $ is an algebraic subgroup of $\widetilde{\mathcal{O}%
}^{\times}/\left(  \underline{1}+\mathcal{F}\right)  $.

We note that every equivalence class in $\pi_{\underline{c}-\underline{1}%
}\left(  \mathcal{O}^{\times}\right)  $ has a polynomial representative, and
then $\pi_{\underline{c}-\underline{1}}\left(  \mathcal{O}^{\times}\right)  $
can be considered an open subset of an affine space, and the multiplication in
$\mathcal{O}^{\times}$ induces a structure of algebraic group in
$\pi_{\underline{c}-\underline{1}}\left(  \mathcal{O}^{\times}\right)  $. In
addition, $\pi_{\underline{c}-\underline{1}}\left(  \widetilde{\mathcal{O}%
}^{\times}\right)  \cong\widetilde{\mathcal{O}}^{\times}/\left(  \underline
{1}+\mathcal{F}\right)  $, as algebraic groups.

We set $\mathcal{J}:=\widetilde{\mathcal{O}}^{\times}/\mathcal{O}^{\times}$.
Every equivalence class has a polynomial representative \ that can be
identified with an element of $J_{\widetilde{\mathcal{O}}}^{\underline
{c}-\underline{1}}$. Each equivalence class depends on $\delta$ coefficients
$a_{i,j}$, see (\ref{delta}), \ $d-1$ of them run over $k^{\times}$ and the
others over $k$. This set of polynomial representatives with the operation
induced by the multiplication in $\widetilde{\mathcal{O}}^{\times}$ is \ a
$k$-algebraic group of dimension $\delta$, more precisely, $\mathcal{J}%
\cong\left(  G_{m}\right)  ^{d-1}\times\left(  G_{a}\right)  ^{\delta-d+1}$
(see \cite[Theorem 11 and its Corollary]{Ros2}, or \cite[Chapter V, Section
17]{Ser}). The group $\mathcal{J}$ appears in the construction of the
generalized Jacobian of a singular curve.

\begin{lemma}
\label{lema2} With the above notation the following identities hold:

\noindent(1) $\left[  \mathcal{J}\right]  =\left(  \mathbb{L}-1\right)
^{d-1}\mathbb{L}^{\delta-d+1}$;

\noindent(2) $\left[  \pi_{\underline{c}-\underline{1}}\left(  \mathcal{O}%
^{\times}\right)  \right]  $ $=\left(  \mathbb{L}-1\right)  \mathbb{L}%
^{\left\Vert \underline{c}\right\Vert -\delta-1}$;

\noindent(3) $\chi_{g}\left(  \mathcal{O}^{\times}\right)  =\left(
\mathbb{L}-1\right)  \mathbb{L}^{-\delta-1}$;

\noindent(4) $\chi_{g}\left(  \mathcal{O}\right)  =\mathbb{L}^{-\delta}$.
\end{lemma}

\begin{proof}
(1) The identity follows from the fact that $\mathcal{J}\cong\left(
k^{\times}\right)  ^{d-1}\times k^{\delta-d+1}$ as algebraic variety, (cf.
\cite[Chapter V, Section 17]{Ser}). (2) From the sequence of algebraic
groups,
\begin{equation}
1\rightarrow\mathcal{O}^{\times}/\left(  \underline{1}+\mathcal{F}\right)
\rightarrow\widetilde{\mathcal{O}}^{\times}/\left(  \underline{1}%
+\mathcal{F}\right)  \rightarrow\mathcal{J}\rightarrow1\text{,} \label{seq1}%
\end{equation}
we have $\left[  \pi_{\underline{c}-\underline{1}}\left(  \mathcal{O}^{\times
}\right)  \right]  $ $=\left[  \mathcal{O}^{\times}/\left(  \underline
{1}+\mathcal{F}\right)  \right]  =\left[  \mathcal{J}\right]  ^{-1}\left[
\widetilde{\mathcal{O}}^{\times}/\left(  \underline{1}+\mathcal{F}\right)
\right]  $. Now, the result follows from (1), since $\left[  \widetilde
{\mathcal{O}}^{\times}/\left(  \underline{1}+\mathcal{F}\right)  \right]
=\left(  \mathbb{L}-1\right)  ^{d}\mathbb{L}^{\left\Vert \underline
{c}\right\Vert -d}$. (3) The third identity follows from (2) by using
$\chi_{g}\left(  \mathcal{O}^{\times}\right)  =\left[  \pi_{\underline
{c}-\underline{1}}\left(  \mathcal{O}^{\times}\right)  \right]  \mathbb{L}%
^{-\left\Vert \underline{c}\right\Vert }$. (4) To prove the last identity we
note that the following exact sequence of (finite dimensional) vector spaces%
\[
0\rightarrow\mathcal{O}/\mathcal{F}\rightarrow\widetilde{\mathcal{O}%
}/\mathcal{F}\rightarrow\widetilde{\mathcal{O}}/\mathcal{O}\rightarrow0
\]
implies that $\left[  \mathcal{O}/\mathcal{F}\right]  =\left[  \widetilde
{\mathcal{O}}/\mathcal{O}\right]  ^{-1}\left[  \widetilde{\mathcal{O}%
}/\mathcal{F}\right]  =\mathbb{L}^{\left\Vert \underline{c}\right\Vert
-\delta}$. Therefore
\[
\chi_{g}\left(  \mathcal{O}\right)  =\left[  \pi_{\underline{c}-\underline{1}%
}\left(  \mathcal{O}\right)  \right]  \mathbb{L}^{-\left\Vert \underline
{c}\right\Vert }=\left[  \mathcal{O}/\mathcal{F}\right]  \mathbb{L}%
^{-\left\Vert \underline{c}\right\Vert }=\mathbb{L}^{-\delta}.
\]

\end{proof}

\section{Zeta Functions for Curve Singularities}

\label{section:zetafunction}

In this section $k$ is a field of characteristic $p\geq0$. \ For
$\underline{n}\in S$ we set
\[
\mathcal{I}_{\underline{n}}:=\left\{  I\subseteq\mathcal{O}\mid I=\underline
{z}\mathcal{O}\text{, with }\underline{v}(\underline{z})=\underline
{n}\right\}  ,
\]
and for $m\in\mathbb{N}$,%
\[
\mathcal{I}_{m}:=%
{\textstyle\bigcup\nolimits_{\substack{\underline{n}\in S\\\left\Vert
\underline{n}\right\Vert =m}}}
\mathcal{I}_{\underline{n}}.
\]

\begin{lemma}
\label{lema3}For any $\underline{n}\in S$, there exists a bijection
$\sigma_{\underline{n}}$ between $\mathcal{I}_{\underline{n}}$ and an
algebraic subset $\sigma_{\underline{n}}\left(  \mathcal{I}_{\underline{n}%
}\right)  $ of $\mathcal{J}$, when $\mathcal{J}$\ is considered as an
algebraic variety. Furthermore, if $\underline{n}\boldsymbol{\geq}%
\underline{c}$, then $\sigma_{\underline{n}}\left(  \mathcal{I}_{\underline
{n}}\right)  =\mathcal{J}$.
\end{lemma}

\begin{proof}
Let $I=\underline{z}\mathcal{O}$ be a principal ideal $\mathcal{I}%
_{\underline{n}}$, with $\underline{z}=$ $\underline{t}^{\underline{n}%
}\underline{\mu}$, $\underline{t}^{\underline{n}}=\left(  t_{1}^{n_{1}}%
,\ldots,t_{d}^{n_{d}}\right)  $ and $\underline{\mu}=\left(  \mu_{1}%
,\ldots,\mu_{d}\right)  \in\widetilde{\mathcal{O}}^{\times}$. Since
$\underline{\mu}$ is determined up to an element of $\mathcal{O}^{\times}$, we
may assume that $\underline{z}=$ $\underline{t}^{\underline{n}}\underline{\mu
}\underline{w}$, with $\underline{\mu}\in\mathcal{J}$ and $\underline{w}%
\in\mathcal{O}^{\times}$. Here we identify $\mathcal{J}$ with a fixed set of
polynomial representatives, and thus $\underline{\mu\ }$ is one of these
representatives. We define
\[%
\begin{array}
[c]{llll}%
\sigma_{\underline{n}}: & \mathcal{I}_{\underline{n}} & \rightarrow &
\mathcal{J}\\
& \underline{t}^{\underline{n}}\underline{\mu}\mathcal{O} & \rightarrow &
\underline{\mu}.
\end{array}
\]
Then $\sigma_{\underline{n}}$\ is a well-defined one-to-one mapping. We now
show that $\sigma_{\underline{n}}\left(  \mathcal{I}_{\underline{n}}\right)  $
is an algebraic subset of $\mathcal{J}$ whose points parametrize the ideals in
$\mathcal{I}_{\underline{n}}$.\ Let $\underline{\mu}$ be a fixed element in
$\mathcal{J}$, if $\underline{t}^{\underline{n}}\underline{\mu}\in\mathcal{O}%
$, then \ $\underline{t}^{\underline{n}}\underline{\mu}$ is the generator of
an ideal in $\mathcal{I}_{\underline{n}}$. The condition `$\underline
{t}^{\underline{n}}\underline{\mu}\in\mathcal{O}$' is algebraic, see
(\ref{delta}), hence $\sigma_{\underline{n}}\left(  \mathcal{I}_{\underline
{n}}\right)  $ is an algebraic subset of $\mathcal{J}$. Finally, if
$\underline{n}\geq\underline{c}$, the condition $\underline{t}^{\underline{n}%
}\underline{\mu}\in\mathcal{O}$ is always true for any $\underline{\mu}%
\in\mathcal{J}$, and then $\sigma_{\underline{n}}\left(  \mathcal{I}%
_{\underline{n}}\right)  =\mathcal{J}$.
\end{proof}

From now on we will identify $\mathcal{I}_{\underline{n}}$\ with
$\sigma_{\underline{n}}\left(  \mathcal{I}_{\underline{n}}\right)  $.

Since
\[
\mathcal{I}_{m}=\cup_{\left\{  \underline{n}\in S\mid\left\Vert \underline
{n}\right\Vert =m\right\}  }\mathcal{I}_{\underline{n}},
\]
by applying the previous lemma, we have that $\mathcal{I}_{m}$\ is an
algebraic subset of $\mathcal{J}$, for any\ $m\in\mathbb{N}$. By using this
fact, the following two formal series are well-defined.

\begin{definition}
\label{def5}We associate to $\mathcal{O}$ the two following zeta functions:
\begin{equation}
Z\left(  T_{1},\ldots,T_{d},\mathcal{O}\right)  :=%
{\textstyle\sum\nolimits_{\underline{n}\in S}}
\left[  \mathcal{I}_{\underline{n}}\right]  \mathbb{L}^{-\left\Vert
\underline{n}\right\Vert }T^{\underline{n}}\in\mathcal{M}_{k} [\![
T_{1},\ldots,T_{d} ]\!] , \label{zeta1}%
\end{equation}
where $T^{\underline{n}}:=T_{1}^{n_{1}} \cdot\ldots\cdot T_{d}^{n_{d}}$, and
\begin{equation}
Z\left(  T,\mathcal{O}\right)  :=Z\left(  T,\ldots,T,\mathcal{O}\right)  .
\label{zeta2}%
\end{equation}

\end{definition}

\begin{lemma}
\label{lema4}The sets $\left\{  \underline{z}\in\mathcal{O}\mid\underline
{v}\left(  \underline{z}\right)  =\underline{n}\right\}  $, $\underline{n}\in
S$, and $\left\{  \underline{z}\in\mathcal{O}\mid\left\Vert \underline
{v}\left(  \underline{z}\right)  \right\Vert =k\right\}  $, $k\in\mathbb{N}%
$,\ are cylindric subsets \ of $\mathcal{O}$. In addition,
\[
\chi_{g}\left(  \left\{  \underline{z}\in\mathcal{O}\mid\underline{v}\left(
\underline{z}\right)  =\underline{n}\right\}  \right)  =\left[  \mathcal{I}%
_{\underline{n}}\right]  \left[  \pi_{\underline{c}-\underline{1}}\left(
\mathcal{O}^{\times}\right)  \right]  \mathbb{L}^{-\left\Vert \underline
{n}+\underline{c}\right\Vert }\text{.}%
\]

\end{lemma}

\begin{proof}
Every $\underline{x}\in\mathcal{O}$, with $\underline{v}\left(  \underline
{x}\right)  =\underline{n}$, can be expressed as%
\begin{align*}
\underline{x}  &  =\underline{t}^{\underline{n}}\underline{\mu}\underline
{w},\text{ }\underline{\mu}\in\mathcal{J}\text{,\ }\underline{w}\in
\mathcal{O}^{\times}\\
&  =\underline{t}^{\underline{n}}\underline{\mu}\pi_{\underline{c}%
-\underline{1}}\left(  \underline{w}\right)  +\underline{t}^{\underline
{n}+\underline{c}}\underline{y},\text{ }\underline{y}\in\widetilde
{\mathcal{O}}.
\end{align*}
Thus $x$ is determined by its $\underline{n}+\underline{c}$ jet, which in turn
is determined by the condition
\[
\underline{\mu}\pi_{\underline{c}-\underline{1}}\left(  \underline{w}\right)
\in\mathcal{I}_{\underline{n}}\times\pi_{\underline{c}-\underline{1}}\left(
\mathcal{O}^{\times}\right)  ,
\]
which is a constructible one. Therefore $\left\{  \underline{z}\in
\mathcal{O}\mid\underline{v}\left(  \underline{z}\right)  =\underline
{n}\right\}  $, $\underline{n}\in S$, is a constructible set and
\[
\chi_{g}\left(  \left\{  \underline{z}\in\mathcal{O}\mid\underline{v}\left(
\underline{z}\right)  =\underline{n}\right\}  \right)  =\left[  \mathcal{I}%
_{\underline{n}}\times\pi_{\underline{c}-\underline{1}}\left(  \mathcal{O}%
^{\times}\right)  \right]  \mathbb{L}^{-\left\Vert \underline{n}+\underline
{c}\right\Vert }\text{.}%
\]

Finally, $\left\{  \underline{z}\in\mathcal{O}\mid\left\Vert \underline
{v}\left(  \underline{z}\right)  \right\Vert =k\right\}  $, $k\in\mathbb{N}$,
is cylindric, since
\[
\left\{  \underline{z}\in\mathcal{O}\mid\left\Vert \underline{v}\left(
\underline{z}\right)  \right\Vert =k\right\}  =%
{\textstyle\bigcup\nolimits_{\left\{  \underline{n}\in S\mid\text{ }\left\Vert
\underline{n}\right\Vert =k\right\}  }}
\left\{  \underline{z}\in\mathcal{O}\mid\underline{v}\left(  \underline
{z}\right)  =\underline{n}\right\}  .
\]

\end{proof}

\begin{corollary}
\label{cor2}With the above notation the following assertions hold:

\noindent(1) the functions
\[%
\begin{array}
[c]{llll}%
T^{\left\Vert \underline{v}\left(  \cdot\right)  \right\Vert }: & \mathcal{O}
& \rightarrow & \mathbb{Z} [\![ T ]\!]\\
& \underline{z} & \rightarrow & T^{\left\Vert \underline{v}\left(
\underline{z}\right)  \right\Vert },
\end{array}
\]

with $T^{\left\Vert \underline{v}\left(  \underline{z}\right)  \right\Vert
}:=0 $, if $\left\Vert \underline{v}\left(  \underline{z}\right)  \right\Vert
=\infty$, and%
\[%
\begin{array}
[c]{llll}%
T^{\underline{v}\left(  \cdot\right)  }: & \mathcal{O} & \rightarrow &
\mathbb{Z} [\![ T_{1},\ldots,T_{d} ]\!]\\
& \underline{z} & \rightarrow & T^{\underline{v}\left(  \underline{z}\right)
},
\end{array}
\]
with $T^{\underline{v}\left(  \underline{z}\right)  }:=0$, if $\left\Vert
\underline{v}\left(  \underline{z}\right)  \right\Vert =\infty$, are cylindric;

\noindent(2) $\left[  \pi_{\underline{c}-\underline{1}}\left(  \mathcal{O}%
^{\times}\right)  \right]  \mathbb{L}^{-\left\Vert \underline{c}\right\Vert
}Z\left(  T_{1},\ldots,T_{d},\mathcal{O}\right)  =%
{\textstyle\int\nolimits_{\mathcal{O}}}
T^{\underline{v}\left(  \underline{z}\right)  }d\chi_{g}$;

\noindent(3) $\left[  \pi_{\underline{c}-\underline{1}}\left(  \mathcal{O}%
^{\times}\right)  \right]  \mathbb{L}^{-\left\Vert \underline{c}\right\Vert
}Z\left(  T,\mathcal{O}\right)  =%
{\textstyle\int\nolimits_{\mathcal{O}}}
T^{\left\Vert \underline{v}\left(  \underline{z}\right)  \right\Vert }%
d\chi_{g}$.
\end{corollary}

\begin{proof}
The assertions follow from \ Definition \ref{def3} by applying the previous lemma.
\end{proof}

Let $J_{\underline{n}}\left(  \mathcal{O}\right)  =\left\{  \underline{z}%
\in\mathcal{O}\mid\underline{v}\left(  \underline{z}\right)  \geqslant
\underline{n}\right\}  $, for $\underline{n}\in\mathbb{N}^{d}$ be an ideal.
Since $J_{\underline{n}}\left(  \mathcal{O}\right)  \subseteq J_{\underline
{n}+\underline{1}}\left(  \mathcal{O}\right)  $, they give a multi-index
filtration of the ring $\mathcal{O}$. Note that the $J_{\underline{n}}\left(
\mathcal{O}\right)  $ are cylindric subsets of $\mathcal{O}$. As in
\cite{CDG4} we introduce the following motivic Poincar\'{e} series.

\begin{definition}
\label{def40} The generalized Poincar\'{e} series of a multi-index filtration
given by the ideals $J_{\underline{n}}\left(  \mathcal{O}\right)  $ is the
integral
\[
P_{g}(T_{1},\ldots,T_{d},\mathcal{O}):=\int_{\mathbb{P}\mathcal{O}%
}T^{\underline{v}(\underline{z})}d\chi_{g}\in\mathcal{M}_{k} [\![ T_{1}%
,\ldots,T_{d} ]\!].
\]

\end{definition}

The generalized Poincar\'e series is related to the zeta function of
Definition \ref{def5} as follows.

\begin{lemma}
\label{rem5} With the above notation:
\[
Z(T_{1},\ldots,T_{d},\mathcal{O})=\mathbb{L}^{\delta+1}P_{g}(T_{1}%
,\ldots,T_{d}).
\]

\end{lemma}

\begin{proof}
By Corollary \ref{cor2} (2), and Lemma \ref{lema2} (2),%
\[
Z(T_{1},\ldots,T_{d},\mathcal{O})=\frac{1}{\left(  \mathbb{L}-1\right)
\mathbb{L}^{-\delta-1}}\int_{\mathcal{O}}T^{\underline{v}(\underline{z})}%
d\chi_{g}=\mathbb{L}^{\delta+1}\int_{\mathbb{P}\mathcal{O}}T^{\underline
{v}(\underline{z})}d\chi_{g},
\]

(cf. Remark \ref{rem1}).
\end{proof}

We set $l(\underline{n}):=\dim_{k}$ $\mathcal{O}/J_{\underline{n}}\left(
\mathcal{O}\right)  $ and the vector $\underline{e}_{i}\in\mathbb{N}^{d}$,
$i=1,\ldots,d$, to have all entries zero except for the $\mathit{i}%
$\textit{-th} one, which is equal to one. Let $I_{0}:=\{1,2,\ldots,d\}$. For
$I\subseteq I_{0}$, let $\#I$ be the number of elements of $I$. Let
$\underline{1}_{I}$ be the element of $\mathbb{N}^{d}$ whose $i$-th component
is equal to $1$ or $0$ if $i\in I$ or $i\notin I$ respectively. Note that
$\underline{0}=\underline{1}_{\emptyset}$ and $\underline{1}=\underline
{1}_{I_{0}}$.

\begin{remark}
\label{nota1}We recall that
\[
\underline{n}\in S\Longleftrightarrow\dim_{k}\text{ }J_{\underline{n}}\left(
\mathcal{O}\right)  /J_{\underline{n}+\underline{e}_{i}}\left(  \mathcal{O}%
\right)  =1\text{, for any }i=1,\ldots,d,
\]
see e.g. \cite{Del2}. Thus, for $\underline{n}\in S$, and for any fixed
$\underline{e}_{i_{0}}$, we have the following exact sequence of $k$-vector
spaces:%
\[
0\rightarrow k\rightarrow J_{\underline{n}+\underline{e}_{i_{0}}}\left(
\mathcal{O}\right)  \rightarrow J_{\underline{n}}\left(  \mathcal{O}\right)
\rightarrow0,
\]
where $J_{\underline{n}}\left(  \mathcal{O}\right)  /J_{\underline
{n}+\underline{e}_{i_{0}}}\left(  \mathcal{O}\right)  \cong k$. Now, if
$\underline{m}\geq\underline{n}+\underline{e}_{i_{0}}+\underline{1}$, from the
previous exact sequence, one gets%
\[
0\rightarrow k\rightarrow J_{\underline{n}+\underline{e}_{i_{0}}}\left(
\mathcal{O}\right)  /t^{\underline{m}+\underline{1}}\widetilde{\mathcal{O}%
}\rightarrow J_{\underline{n}}\left(  \mathcal{O}\right)  /t^{\underline
{m}+\underline{1}}\widetilde{\mathcal{O}}\rightarrow0,
\]
and hence%
\[
\left[  J_{\underline{n}}\left(  \mathcal{O}\right)  /t^{\underline
{m}+\underline{1}}\widetilde{\mathcal{O}}\right]  =\mathbb{L}\left[
J_{\underline{n}+\underline{e}_{i_{0}}}\left(  \mathcal{O}\right)
/t^{\underline{m}+\underline{1}}\widetilde{\mathcal{O}}\right]  .
\]

\end{remark}

\begin{proposition}
\label{lema8}$\left[  \mathcal{I}_{\underline{n}}\right]  =\left(
\mathbb{L}-1\right)  ^{-1}\mathbb{L}^{\left\Vert \underline{n}\right\Vert +1}%
{\textstyle\sum\limits_{I\subseteq I_{0}}}
\left(  -1\right)  ^{\#(I)}\mathbb{L}^{-l\left(  \underline{n}+\underline
{1}_{I}\right)  }$, for $\underline{n}\in S$.
\end{proposition}

\begin{proof}
We claim that%
\begin{equation}
\chi_{g}\left(  J_{\underline{n}}\left(  \mathcal{O}\right)  \right)
=\left\{
\begin{array}
[c]{lll}%
\mathbb{L\cdot}\chi_{g}\left(  J_{\underline{n}+\underline{e}_{i_{0}}}\left(
\mathcal{O}\right)  \right)  , & \text{if } & \dim_{k}\left(  J_{\underline
{n}}\left(  \mathcal{O}\right)  /J_{\underline{n}+\underline{e}_{i_{0}}%
}\left(  \mathcal{O}\right)  \right)  =1;\\
&  & \\
\chi_{g}\left(  J_{\underline{n}+\underline{e}_{i_{0}}}\left(  \mathcal{O}%
\right)  \right)  , & \text{if} & \dim_{k}\left(  J_{\underline{n}}\left(
\mathcal{O}\right)  /J_{\underline{n}+\underline{e}_{i_{0}}}\left(
\mathcal{O}\right)  \right)  =0,
\end{array}
\right.  \label{S2}%
\end{equation}
for any $\underline{e}_{i_{0}}$. The formula is clear if $\dim_{k}\left(
J_{\underline{n}}\left(  \mathcal{O}\right)  /J_{\underline{n}+\underline
{e}_{i_{0}}}\left(  \mathcal{O}\right)  \right)  =0$, i.e., if $J_{\underline
{n}}\left(  \mathcal{O}\right)  =J_{\underline{n}+\underline{e}_{i_{0}}%
}\left(  \mathcal{O}\right)  $; thus we can assume that $\dim_{k}\left(
J_{\underline{n}}\left(  \mathcal{O}\right)  /J_{\underline{n}+\underline
{e}_{i_{0}}}\left(  \mathcal{O}\right)  \right)  =1$, i.e. $\underline{n}\in
S$. By taking $\underline{m}$ as in Remark \ref{nota1},\ one gets
\begin{align*}
\chi_{g}\left(  J_{\underline{n}}\left(  \mathcal{O}\right)  \right)   &
=\mathbb{L}^{\mathbb{-}\left\Vert \underline{m}+\underline{1}\right\Vert
}\left[  J_{\underline{n}}\left(  \mathcal{O}\right)  /t^{\underline
{m}+\underline{1}}\widetilde{\mathcal{O}}\right]  =\mathbb{L}\left(
\mathbb{L}^{\mathbb{-}\left\Vert \underline{m}+\underline{1}\right\Vert
}\left[  J_{\underline{n}+\underline{e}_{i_{0}}}\left(  \mathcal{O}\right)
/t^{\underline{m}+\underline{1}}\widetilde{\mathcal{O}}\right]  \right) \\
&  =\mathbb{L} \cdot\chi_{g}\left(  J_{\underline{n}+\underline{e}_{i_{0}}%
}\left(  \mathcal{O}\right)  \right)  .
\end{align*}

Now we fix a sequence of the form%
\[
\underline{0}=\underline{m}_{0}\leqslant\underline{m}_{1}\leqslant
\ldots\leqslant\underline{m}_{j}\leqslant\underline{m}_{j+1}\leqslant
\ldots\leqslant\underline{m}_{k}=\underline{n},
\]
where $\underline{m}_{j+1}=\underline{m}_{j}+\underline{e}_{j_{i}}$, for
$j=0,\ldots,k-1$. Then by applying (\ref{S2}) we have%
\begin{equation}
\chi_{g}\left(  J_{\underline{n}}\left(  \mathcal{O}\right)  \right)
=\mathbb{L}^{-l\left(  \underline{n}\right)  }\cdot\chi_{g}\left(
\mathcal{O}\right)  . \label{S3}%
\end{equation}
On the other hand,%
\begin{align*}
\chi_{g}\left(  \left\{  \underline{z}\in\mathcal{O}\mid\underline{v}\left(
\underline{z}\right)  =\underline{n}\right\}  \right)   &  =\chi_{g}\left(
J_{\underline{n}}\left(  \mathcal{O}\right)  \setminus%
{\textstyle\bigcup\limits_{\boldsymbol{i}=1}^{d}}
J_{\underline{n}+\underline{e}_{i}}\left(  \mathcal{O}\right)  \right) \\
&  =\chi_{g}\left(  J_{\underline{n}}\left(  \mathcal{O}\right)  \right)
-\chi_{g}\left(
{\textstyle\bigcup\limits_{\boldsymbol{i}=1}^{d}}
J_{\underline{n}+\underline{e}_{i}}\left(  \mathcal{O}\right)  \right)  .
\end{align*}
Now by using the identities%
\[
\chi_{g}\left(
{\textstyle\bigcup\limits_{i=1}^{n}}
A_{i}\right)  =%
{\textstyle\sum\limits_{\substack{J\subseteq\left\{  1,2,\ldots,n\right\}
\\J\neq\emptyset}}}
\left(  -1\right)  ^{\#(J)-1}\chi_{g}\bigg(
{\textstyle\bigcap\limits_{j\in J}}
A_{j}\bigg)  ,
\]%
\[
J_{\underline{n}+\underline{e}_{i_{1}}}\left(  \mathcal{O}\right)  \cap
\ldots\cap J_{\underline{n}+\underline{e}_{i_{j}}}\left(  \mathcal{O}\right)
=J_{\underline{n}+\underline{e}_{i_{1}}+\ldots+\underline{e}_{i_{j}}}\left(
\mathcal{O}\right)  ,
\]
\ (\ref{S3}) and Lemma \ref{lema4},\ we obtain
\begin{align*}
\left[  \mathcal{I}_{\underline{n}}\right]   &  =\frac{\mathbb{L}^{\left\Vert
\underline{n}+\underline{c}\right\Vert }}{\left[  \pi_{\underline
{c}-\underline{1}}\left(  \mathcal{O}^{\times}\right)  \right]  }%
\bigg( \chi_{g}\left(  J_{\underline{n}}\left(  \mathcal{O}\right)  \right)  -%
{\textstyle\sum\limits_{\substack{I\subseteq\left\{  1,2,\ldots,d\right\}
\\I\neq\emptyset}}}
\left(  -1\right)  ^{\#(I)-1}\chi_{g}\left(  J_{\underline{n}+\sum_{i\in
I}\underline{e}_{i}}\left(  \mathcal{O}\right)  \right)  \bigg)\\
&  =\frac{\mathbb{L}^{\left\Vert \underline{n}+\underline{c}\right\Vert }%
\chi_{g}\left(  \mathcal{O}\right)  }{\left[  \pi_{\underline{c}-\underline
{1}}\left(  \mathcal{O}^{\times}\right)  \right]  }\bigg( \mathbb{L}%
^{-l\left(  \underline{n}\right)  }-%
{\textstyle\sum\limits_{\substack{I\subseteq\left\{  1,2,\ldots,d\right\}
\\I\neq\emptyset}}}
\left(  -1\right)  ^{\#(I)-1}\mathbb{L}^{-l\left(  \underline{n}+\underline
{1}_{I}\right)  }\bigg) .
\end{align*}

Finally, the result follows from the previous identity by using
\[
\left[  \pi_{\underline{c}-\underline{1}}\left(  \mathcal{O}^{\times}\right)
\right]  =\left(  \mathbb{L}-1\right)  \mathbb{L}^{\left\Vert \underline
{c}\right\Vert -\delta-1}\text{and }\chi_{g}\left(  \mathcal{O}\right)
=\mathbb{L}^{-\delta},
\]
(cf. Lemma \ref{lema2}).
\end{proof}

\begin{remark}
\label{nota2}Let $k$ be a field of characteristic $p>0$. Let $Y$ be an
algebraic curve defined over $k$. Let $\mathcal{O}_{P,Y}$ be the local ring of
$Y$ at the point $P$, and $\widehat{O}_{P,Y}$\ its completion. Then
$\mathcal{J}\cong\left(  G_{m}\right)  ^{d-1}\times\Gamma$, where $\Gamma$ is
a subgroup of a product of groups of Witt vectors of finite length. If $p\geq
c_{i}$, for $i=1,\ldots,d$, where $\underline{c}=\left(  c_{1},\ldots
,c_{d}\right)  $ is the conductor of the semigroup of $\widehat{O}_{P,Y}$,
then $\mathcal{J}\cong\left(  G_{m}\right)  ^{d-1}\times\left(  G_{a}\right)
^{\delta-d+1}$ (cf. \cite[Proposition 9, Chapter V, Sections 16 ]{Ser}).\ We
can attach to $\widehat{O}_{P,Y}$ a zeta function $Z\left(  T_{1},\ldots
,T_{d},\widehat{O}_{P,Y}\right)  $ defined as before. All the results
presented so far are valid in this context, in particular Proposition
\ref{lema8}.
\end{remark}

\section{\label{rationality}Rationality of $Z\left(  T_{1},\ldots
,T_{d},\mathcal{O}\right)  $}

From now on $k$ is a field of characteristic $p\geq0$, and $\mathcal{O}$ is a
totally rational ring as before. The aim of this section is to prove the
rationality of the zeta function $Z(T_{1},\ldots,T_{d},\mathcal{O})$ and,
subsequently, of the generalized Poincar\'{e} series $P_{g}(T_{1},\ldots
,T_{d}, \mathcal{O})$ by Lemma \ref{rem5}, giving also an explicit formula for it.

We start establishing the notation and preliminary results required in the proof.

We set $I_{0}:=\left\{  1,2,\ldots,d\right\}  $ and for a \ subset $J$ of
$I_{0}$,
\[
H_{J}:=\left\{  \underline{n}\in S\mid n_{j}\geqslant c_{j}\Leftrightarrow
j\in J\right\}  ,
\]
where $\underline{c}=\left(  c_{1},\ldots,c_{d}\right)  $ is the conductor of
$S$, and also%
\[
H_{J}\left(  \mathcal{O}\right)  :=\left\{  \underline{z}\in\mathcal{O}%
\mid\underline{v}\left(  \underline{z}\right)  \in H_{J}\right\}  .
\]
Note that $H_{\emptyset}\left(  \mathcal{O}\right)  =\left\{  \underline{z}%
\in\mathcal{O}\mid0\leqslant v\left(  z_{i}\right)  \leqslant c_{i}-1\text{,
}i=1,\ldots,d\right\}  $, and $H_{I_{0}}\left(  \mathcal{O}\right)
=\mathcal{F}$. \ Given $\underline{m}\in\mathbb{N}^{d}$ such that
$\underline{c}>\underline{m}$, i.e., $c_{i}>m_{i}$, for $i=1,\ldots,d$, we
set
\[
H_{J,\underline{m}}:=\left\{  \underline{n}\in S\mid n_{j}\geqslant
c_{j}\text{ if\ }j\in J\text{, and }n_{j}=m_{j}\text{ if }j\notin J\right\}
,
\]
\[
H_{J,\underline{m}}\left(  \mathcal{O}\right)  :=\left\{  \underline{z}%
\in\mathcal{O}\mid\underline{v}\left(  \underline{z}\right)  \in
H_{J,\underline{m}}\right\}  ,
\]
and for a fixed $J$ satisfying $\emptyset\subsetneq J\subsetneq I_{0}$,
\[
B_{J}:=\left\{  \underline{m}\in\mathbb{N}^{\#J}\mid H_{J,\underline{m}}%
\neq\emptyset\right\}  .
\]
Therefore for $\emptyset\subsetneq J\subsetneq I_{0}$, one gets the following
partition for $H_{J}\left(  \mathcal{O}\right)  $:
\begin{equation}
H_{J}\left(  \mathcal{O}\right)  =%
{\textstyle\bigcup\limits_{\underline{m}\in B_{J}}}
H_{J,\underline{m}}\left(  \mathcal{O}\right)  . \label{H_formula}%
\end{equation}

\begin{lemma}
\label{lema6}With the above notation the following assertions hold:

\noindent(1) Let $J=\left\{  1,\ldots,r\right\}  $ with $1\leqslant r<d$ and
\ let $\underline{m}\in\mathbb{N}^{d}$ such that $\underline{c}>\underline{m}%
$. If $H_{J,\underline{m}}\neq\emptyset$, then \
\[
H_{J,\underline{m}}=\left\{  \underline{n}\in\mathbb{N}^{d}\mid%
\begin{array}
[c]{ll}%
n_{i}\geqslant c_{i}\text{,} & \text{for }i=1,\ldots,r\text{, and}\\
n_{i}=m_{i}, & \text{for }i=r+1,\ldots,d
\end{array}
\right\}  ;
\]

\noindent(2) $H_{J,\underline{m}}\left(  \mathcal{O}\right)  $ and
$H_{J}\left(  \mathcal{O}\right)  $ are cylindric subsets of $\mathcal{O}$.
\end{lemma}

\begin{proof}
(1) Since $H_{J,\underline{m}}\neq\emptyset$, there exist $\underline
{f}(\underline{m}):=\left(  e_{1},\ldots,e_{r},m_{r+1},\ldots,m_{d}\right)
\in H_{J,\underline{m}}$ and $\underline{z}=\left(  z_{1},\ldots,z_{d}\right)
\in\mathcal{O}$ such that
\[
z_{i}=\left\{
\begin{array}
[c]{lll}%
{\textstyle\sum\nolimits_{k=e_{i}}^{\infty}}
a_{k,i}t_{i}^{k},\  & \text{with }a_{e_{i},i}\neq0\text{,} & \text{for
}i=1,\ldots,r;\\
&  & \\%
{\textstyle\sum\nolimits_{k=m_{i}}^{\infty}}
a_{k,i}t_{i}^{k}, & \text{with }a_{m_{i},i}\neq0\text{,} & \text{for
}i=r+1,\ldots,d.
\end{array}
\right.
\]
Since $\mathcal{O}$ is a cylindric subset of $\widetilde{\mathcal{O}}$ defined
by the condition $\Delta=0$ (see (\ref{delta})), that involves only the
variables $a_{k,i}$ with $0\leqslant k<c_{k}$, $k=1,\ldots,d$, it follows that
any $\underline{y}=\left(  y_{1},\ldots,y_{d}\right)  \in\widetilde
{\mathcal{O}}$ of the form%

\[
y_{i}=\left\{
\begin{array}
[c]{lll}%
{\textstyle\sum\nolimits_{k=c_{i}}^{\infty}}
a_{k,i}t_{i}^{k},\  & \text{for }i=1,\ldots,r; & \\
&  & \\%
{\textstyle\sum\nolimits_{k=m_{i}}^{\infty}}
a_{k,i}t_{i}^{k}, & \text{with }a_{m_{i},i}\neq0\text{,} & \text{for
}i=r+1,\ldots,d,
\end{array}
\right.
\]
belongs to $\mathcal{O}$, and therefore
\[
H_{J,\underline{m}}=\left\{  \underline{n}\in\mathbb{N}^{d}\mid%
\begin{array}
[c]{ll}%
n_{i}\geqslant c_{i}\text{,} & \text{for }i=1,\ldots,r\text{, and}\\
n_{i}=m_{i}, & \text{for }i=r+1,\ldots,d.
\end{array}
\right\}  .
\]

(2) Since $H_{J}\left(  \mathcal{O}\right)  $ is a finite disjoint union of
subsets of the form $H_{J,\underline{m}}\left(  \mathcal{O}\right)  $, it is
sufficient to show that $H_{J,\underline{m}}\left(  \mathcal{O}\right)  $ is a
cylindric subset of $\mathcal{O}$. On the other hand, since $H_{J,\underline
{m}}\left(  \mathcal{O}\right)  =\mathcal{O}\cap\left\{  \underline{z}%
\in\widetilde{\mathcal{O}}\mid\underline{v}\left(  \underline{z}\right)  \in
H_{J,\underline{m}}\right\}  $ and $\mathcal{O}$\ is a cylindric subset of
$\widetilde{\mathcal{O}}$, it is enough to show that $\underline{v}\left(
\underline{z}\right)  \in H_{J,\underline{m}}$ is a constructible condition in
$J_{\widetilde{\mathcal{O}}}^{\underline{l}}$, for some $\underline{l}%
\in\mathbb{N}^{d}$. Let $\underline{l}=\left(  c_{1},\ldots,c_{r}%
,m_{r+1}+1,\ldots,m_{d}+1\right)  $, and let \ $\underline{z}=\left(
z_{1},\ldots,z_{d}\right)  \in\widetilde{\mathcal{O}}$ with $z_{i}=%
{\textstyle\sum\nolimits_{k=0}^{\infty}}
a_{k,i}t_{i}^{k}$, \ for $i=1,\ldots,d$. Since%
\[
v_{i}\left(  z_{i}\right)  =m_{i}\Leftrightarrow\left\{
\begin{array}
[c]{ll}%
a_{k,i}=0\text{,} & k=0,\ldots,m_{i}-1;\text{ }\\
& \\
a_{m_{i},i}\neq0\text{,} &
\end{array}
\right.
\]
and \
\[
v_{i}\left(  z_{i}\right)  \geqslant c_{i}\Leftrightarrow\left\{
\begin{array}
[c]{ll}%
a_{k,i}=0\text{,} & k=0,\ldots,c_{i}-1,\text{ }%
\end{array}
\right.
\]
thus $\underline{v}\left(  \underline{z}\right)  \in H_{J,\underline{m}}$ is a
constructible condition in $J_{\widetilde{\mathcal{O}}}^{\underline{l}}$.
\end{proof}

\begin{remark}
\label{rem2}Let $J=\left\{  1,\ldots,r\right\}  $ with $1\leqslant r<d$ and
\ let $\underline{m}\in\mathbb{N}^{d}$ such that $\underline{c}>\underline{m}%
$. If $H_{J,\underline{m}}\neq\emptyset$, then $\left[  \mathcal{I}%
_{\underline{k}}\right]  =\left[  \mathcal{I}_{\underline{f_{J}}(\underline
{m})} \right]  $, with $\underline{f_{J}}(\underline{m})=\left(  c_{1}%
,\ldots,c_{r},m_{r+1},\ldots,m_{d}\right)  $, for any $\underline{k}\in
H_{J,\underline{m}}$.

The remark follows from the following observation. With the notation used in
the proof of Lemma \ref{lema3}, the following conditions are equivalent:%
\begin{align*}
\sigma_{\underline{k}}\left(  I\right)   &  =\underline{\mu},\underline{\text{
}k}\in H_{J,\underline{m}}\text{ }\Leftrightarrow t^{\underline{k}}%
\underline{\mu}\underline{v}\in\mathcal{O}\text{, for any }\underline{v}%
\in\mathcal{O}^{\times}\text{, }\underline{k}\in H_{J,\underline{m}}\\
&  \Leftrightarrow t^{\underline{f_{J}}(\underline{m})}\underline{\mu
}\underline{v}\in\mathcal{O}\text{, for any }\underline{v}\in\mathcal{O}%
^{\times}.
\end{align*}
In the proof of the last equivalence we use the same reasoning as that used in
the proof of Lemma \ref{lema6} (1).
\end{remark}

\begin{lemma}
\label{lema7}Let $J$ be a non-empty and proper subset of $I_{0}$, such that
$H_{J,\underline{m}}\left(  \mathcal{O}\right)  \neq\emptyset$. Then%
\[%
{\textstyle\int\limits_{H_{J,\underline{m}}\left(  \mathcal{O}\right)  }}
T^{\underline{v}\left(  \underline{z}\right)  }d\chi_{g}=\frac{\left[
\mathcal{I}_{\underline{f_{J}}(\underline{m})}\right]  \left[  \pi
_{\underline{c}-\underline{1}}\left(  \mathcal{O}^{\times}\right)  \right]
\mathbb{L}^{-\left\Vert \underline{c}\right\Vert -\left\Vert \underline{f_{J}%
}(\underline{m})\right\Vert }T^{\underline{f_{J}}(\underline{m})}}{%
{\textstyle\prod\limits_{i=1}^{r}}
\left(  1-\mathbb{L}^{-1}T_{i}\right)  }\text{,}%
\]
where $\underline{f_{J}}(\underline{m})=\left(  c_{1},\ldots,c_{r}%
,m_{r+1},\ldots,m_{d}\right)  \in S$, with $m_{i}<c_{i}$, $r+1\leqslant i\leq
d$.
\end{lemma}

\begin{proof}
Without loss of generality we assume that $J=\left\{  1,\ldots,r\right\}  $,
with $1\leqslant r<d$. With this notation, by using $H_{J,\underline{m}%
}\left(  \mathcal{O}\right)  \neq\emptyset$ and Lemma \ref{lema6} (1), \ we have%

\[
H_{J,\underline{m}}=\left\{  \underline{n}\in\mathbb{N}^{d}\mid%
\begin{array}
[c]{ll}%
n_{i}\geqslant c_{i}\text{,} & \text{for }i=1,\ldots,r\text{, and}\\
n_{i}=m_{i}, & \text{for }i=r+1,\ldots,d.
\end{array}
\right\}
\]
Now, by using Lemma \ref{lema4}\ and Remark \ref{rem2} we have%
\begin{align*}%
{\textstyle\int\limits_{H_{J,\underline{m}}\left(  \mathcal{O}\right)  }}
T^{\underline{v}\left(  \underline{z}\right)  }d\chi_{g}  &  =\left[
\mathcal{I}_{\underline{f_{J}}(\underline{m})}\right]  \left[  \pi
_{\underline{c}-\underline{1}}\left(  \mathcal{O}^{\times}\right)  \right]
\mathbb{L}^{-\left\Vert \underline{c}\right\Vert -\left\Vert \underline{f_{J}%
}(\underline{m})\right\Vert }T^{\underline{f_{J}}(\underline{m})}\left(
{\textstyle\sum\limits_{e\in\mathbb{N}^{r}}}
\mathbb{L}^{-\left\Vert \underline{e}\right\Vert }T^{\underline{e}}\right) \\
&  =\frac{\left[  \mathcal{I}_{\underline{f_{J}}(\underline{m})}\right]
\left[  \pi_{\underline{c}-\underline{1}}\left(  \mathcal{O}^{\times}\right)
\right]  \mathbb{L}^{-\left\Vert \underline{c}\right\Vert -\left\Vert
\underline{f_{J}}(\underline{m})\right\Vert }T^{\underline{f_{J}}%
(\underline{m})}}{%
{\textstyle\prod\limits_{i=1}^{r}}
\left(  1-\mathbb{L}^{-1}T_{i}\right)  },
\end{align*}
where $\underline{f_{J}}(\underline{m})=\left(  c_{1},\ldots,c_{r}%
,m_{r+1},\ldots,m_{d}\right)  \in S$, with $m_{i}<c_{i}$, $r+1\leqslant i\leq
d$.
\end{proof}

\begin{theorem}
\label{teo1}Let $k$ be a field of characteristic $p\geq0$, and $\mathcal{O}%
$\ a totally rational ring as before. Then \ (1)%
\[
Z\left(  T_{1},\ldots,T_{d},\mathcal{O}\right)  =%
{\textstyle\sum\limits_{%
\begin{array}
[c]{c}%
\underline{n}\in S\\
\underline{0}\leq\underline{n}<\underline{c}%
\end{array}
}}
\left[  \mathcal{I}_{\underline{n}}\right]  \mathbb{L}^{-\left\Vert
\underline{n}\right\Vert }T^{\underline{n}}%
\]
\begin{align*}
&  +%
{\textstyle\sum\limits_{\emptyset\subsetneq J\subsetneq I_{0}}}
\text{ \ }%
{\textstyle\sum\limits_{\substack{\underline{m}\in B_{J}}}}
\left[  \mathcal{I}_{\underline{f_{J}}(\underline{m})}\right]  \left[
\pi_{\underline{c}-\underline{1}}\left(  \mathcal{O}^{\times}\right)  \right]
\mathbb{L}^{-\left\Vert \underline{c}\right\Vert -\left\Vert \underline{f_{J}%
}(\underline{m})\right\Vert }\frac{T^{\underline{f_{J}}(\underline{m})}}{%
{\textstyle\prod\limits_{i=1}^{r_{J}}}
\left(  1-\mathbb{L}^{-1}T_{i}\right)  }\\
&  +\left[  \mathcal{J}\right]  \mathbb{L}^{-\left\Vert \underline
{c}\right\Vert }\frac{T^{\underline{c}}}{%
{\textstyle\prod\limits_{i=1}^{d}}
\left(  1-\mathbb{L}^{-1}T_{i}\right)  },
\end{align*}
where $\underline{f_{J}}(\underline{m})=\left(  c_{1},\ldots,c_{r_{J}%
},m_{r_{J}+1},\ldots,m_{d}\right)  \in S$, with $m_{i}<c_{i}$, $r_{J}%
+1\leqslant i\leq d$, and $1\leqslant r_{J}<d$.

(2)%
\[
Z\left(  T_{1},\ldots,T_{d},\mathcal{O}\right)  =\frac{M\left(  T_{1}%
,\ldots,T_{d},\mathcal{O}\right)  }{%
{\textstyle\prod\limits_{i=1}^{d}}
\left(  1-\mathbb{L}^{-1}T_{i}\right)  }%
\]
where $M\left(  T_{1},\ldots,T_{d},\mathcal{O}\right)  $\ is a polynomial in
$\mathcal{M}_{k}\left[  T_{1},\ldots,T_{d}\right]  $ of degree at most
$\left\Vert \underline{c}\right\Vert $ that satisfies $M\left(  \mathbb{L}%
,\ldots,\mathbb{L},\mathcal{O}\right)  =\left[  \mathcal{J}\right]  $.
\end{theorem}

\begin{proof}
Since $Z\left(  T_{1},\ldots,T_{d},\mathcal{O}\right)  =\left[  \pi
_{\underline{c}-\underline{1}}\left(  \mathcal{O}^{\times}\right)  \right]
^{-1}\mathbb{L}^{\left\Vert \underline{c}\right\Vert }%
{\textstyle\int\nolimits_{\mathcal{O}}}
T^{\underline{v}\left(  \underline{z}\right)  }d\chi_{g}$ (cf. Corollary
\ref{cor2} (2)) and $\mathcal{O}=\cup_{J\subseteq I_{0}}H_{J}\left(
\mathcal{O}\right)  $ is a disjoint union of cylindric subsets (cf. Lemma
\ref{lema6} (2)), $Z\left(  T_{1},\ldots,T_{d},\mathcal{O}\right)  $ is equal
to a finite sum of integrals of type
\[
Z_{H_{J}}\left(  T_{1},\ldots,T_{d},\mathcal{O}\right)  :=\left[
\pi_{\underline{c}-\underline{1}}\left(  \mathcal{O}^{\times}\right)  \right]
^{-1}\mathbb{L}^{\left\Vert \underline{c}\right\Vert }%
{\textstyle\int\limits_{H_{J}\left(  \mathcal{O}\right)  }}
T^{\underline{v}\left(  \underline{z}\right)  }d\chi_{g} .
\]

In the case in which $J=\emptyset$,
\[
Z_{H_{\emptyset}}\left(  T_{1},\ldots,T_{d},\mathcal{O}\right)  =%
{\textstyle\sum\limits_{%
\begin{array}
[c]{c}%
\underline{n}\in S\\
\underline{0}\leq\underline{n}<\underline{c}%
\end{array}
}}
\left[  \mathcal{I}_{\underline{n}}\right]  \mathbb{L}^{-\left\Vert
\underline{n}\right\Vert }T^{\underline{n}}\in\mathcal{M}_{k}\left[
T_{1},\ldots,T_{d}\right]  ,
\]
and the degree of $Z_{H_{\emptyset}}\left(  T_{1},\ldots,T_{d},\mathcal{O}%
\right)  $ is less than or equal to $\left\Vert \underline{c}\right\Vert -d$.

In the case $J=I_{0}$, by using Lemma \ref{lema3}, we have
\[
Z_{H_{I_{0}}}\left(  T_{1},\ldots,T_{d},\mathcal{O}\right)  =\left[
\mathcal{J}\right]  \mathbb{L}^{-\left\Vert \underline{c}\right\Vert }%
\frac{T^{\underline{c}}}{%
{\textstyle\prod\limits_{i=1}^{d}}
\left(  1-\mathbb{L}^{-1}T_{i}\right)  }.
\]

In the case in which $\emptyset\subsetneq J\subsetneq I_{0}$, we use the fact
that $H_{J}\left(  \mathcal{O}\right)  $ is a finite disjoint union of
cylindric sets of the form $H_{J,\underline{m}}\left(  \mathcal{O}\right)  $
(cf. (\ref{H_formula})) to reduce the problem to the computation of the
following integral:
\begin{align*}
Z_{H_{J,\underline{m}}}\left(  T_{1},\ldots,T_{d},\mathcal{O}\right)   &
:=\left[  \pi_{\underline{c}-\underline{1}}\left(  \mathcal{O}^{\times
}\right)  \right]  ^{-1}\mathbb{L}^{\left\Vert \underline{c}\right\Vert }%
{\textstyle\int\limits_{H_{J,\underline{m}}\left(  \mathcal{O}\right)  }}
T^{\underline{v}\left(  \underline{z}\right)  }d\chi_{g}\\
&  =\frac{\left[  \mathcal{I}_{\underline{f_{J}}(\underline{m})}\right]
\mathbb{L}^{-\left\Vert \underline{f_{J}}(\underline{m})\right\Vert
}T^{\underline{f_{J}}(\underline{m})}}{%
{\textstyle\prod\limits_{i=1}^{r_{J}}}
\left(  1-\mathbb{L}^{-1}T_{i}\right)  },
\end{align*}
(cf. Lemma \ref{lema7}), where $\underline{f_{J}}(\underline{m})=\left(
c_{1},\ldots,c_{r_{J}},m_{r_{J}+1},\ldots,m_{d}\right)  \in S$, with
$m_{i}<c_{i}$, $r_{J}+1\leqslant i\leq d$, and $1\leqslant r_{J}<d$.\ Now the
announced explicit formula follows from the previous discussion, and the
second part of the theorem is a straight consequence of it.
\end{proof}

\begin{corollary}
\label{cor3}The zeta function $Z\left(  T,\mathcal{O}\right)  $ is a rational
function of the form%
\[
Z\left(  T,\mathcal{O}\right)  =\frac{R\left(  T,\mathcal{O}\right)  }{\left(
1-\mathbb{L}^{-1}T\right)  ^{d}},
\]
where $R\left(  T,\mathcal{O}\right)  $\ is a polynomial in $\mathcal{M}%
_{k}\left[  T\right]  $ of degree at most $\left\Vert \underline{c}\right\Vert
$ that satisfies $R\left(  \mathbb{L},\mathcal{O}\right)  =\left[
\mathcal{J}\right]  $.
\end{corollary}

\begin{corollary}
\label{coroPg}The generalized Poincar\'{e} series is a rational function of
the form%
\[
P_{g}\left(  T_{1},\ldots,T_{d},\mathcal{O}\right)  =\frac{Q\left(
T_{1},\ldots,T_{d},\mathcal{O}\right)  }{%
{\textstyle\prod\limits_{i=1}^{d}}
\left(  1-\mathbb{L}^{-1}T_{i}\right)  },
\]
where $Q\left(  T_{1},\ldots,T_{d},\mathcal{O}\right)  $\ is a polynomial in
$\mathcal{M}_{k}\left[  T_{1},\ldots,T_{d}\right]  $ of degree at most
$\left\Vert \underline{c}\right\Vert $ that satisfies $Q\left(  \mathbb{L}%
,\ldots,\mathbb{L}, \mathcal{O} \right)  =\mathbb{L}^{-\delta-1}\left[
\mathcal{J}\right]  $.
\end{corollary}

\begin{definition}
Let $k$ be a field of characteristic $p\geq0$. Let $\mathcal{O}=\widehat
{O}_{P,Y}$, where $Y$ is an algebraic curve over $k$, and $P$ is a singular
point of $Y$. We say that $k$ is big enough for $Y$, if for every singular
point $P$ in $Y$ the following two conditions hold: 1) $\mathcal{O}$ is
totally rational and 2) $\mathcal{J}\cong\left(  G_{m}\right)  ^{d-1}%
\times\left(  G_{a}\right)  ^{\delta-d+1}$.
\end{definition}

Note that by Remark \ref{nota2}, the condition `$k$ is big enough for $Y$' is
fulfilled when $p$ is big enough.

\begin{corollary}
\label{cor4}Let $k$ be a field of characteristic $p\geq0$. Let $\mathcal{O}%
=\widehat{O}_{P,Y}$ where $Y$ is an algebraic curve over $k$, and $P$ is a
singular point of $Y$. If $k$ is big enough for $Y$, then $Z\left(
T_{1},\ldots,T_{d},\mathcal{O}\right)  $ is completely determined by the
semigroup of $\mathcal{O}$.
\end{corollary}

\begin{proof}
By the explicit formula of Theorem \ref{teo1}, $Z\left(  T_{1},\ldots
,T_{d},\mathcal{O}\right)  $ is a rational function in the variables
$T_{1},\ldots,T_{d}$, and $\mathbb{L}$, depending on $S$, $\left[
\pi_{\underline{c}-\underline{1}}\left(  \mathcal{O}^{\times}\right)  \right]
$, $\left[  \mathcal{J}\right]  $, and $\left[  \mathcal{I}\underline{_{m}%
}\right]  $ for $\left\Vert \underline{m}\right\Vert <\left\Vert \underline
{c}\right\Vert $. In characteristic zero, $S$ determines uniquely $\left[
\pi_{\underline{c}-\underline{1}}\left(  \mathcal{O}^{\times}\right)  \right]
$, $\left[  \mathcal{J}\right]  $, $\left[  \mathcal{I}\underline{_{m}%
}\right]  $ for $\left\Vert \underline{m}\right\Vert <\left\Vert \underline
{c}\right\Vert $ (cf. Lemma \ref{lema2} and Proposition \ref{lema8}). If the
characteristic is $p>0$, the hypothesis \textquotedblleft$k$ is big enough for
$Y$\textquotedblright\ is required to assure that $\left[  \mathcal{J}\right]
$ is determined by the semigroup of $\mathcal{O}$.
\end{proof}

\section{Additive Invariants and Specialization of Zeta Functions}

\begin{definition}
\label{universalzetafunction} Put $k=\mathbb{C}$. Consider a semigroup
$S\subset\mathbb{N}^{d}$, such that $S=S\left(  \mathcal{O}\right)  $ for some
$\mathcal{O}=\widehat{O}_{X,P}$ where $X$ is an algebraic curve over
$\mathbb{C}$, and $P$ is a singular point of $X$. We set%

\[
\mathcal{I}_{\underline{n}}\left(  U\right)  :=\left(  U-1\right)
^{-1}U^{\left\Vert \underline{n}\right\Vert +1}%
{\textstyle\sum\limits_{I\subseteq I_{0}}}
\left(  -1\right)  ^{\#(I)}U^{-l\left(  \underline{n}+\underline{1}%
_{I}\right)  }\text{, for }\underline{n}\in S,
\]
and%
\[
\mathcal{Z}\left(  T_{1},\ldots,T_{d},U,S\right)  :=%
{\textstyle\sum\limits_{%
\begin{array}
[c]{c}%
\underline{n}\in S\\
\underline{0}\leq\underline{n}<\underline{c}%
\end{array}
}}
\mathcal{I}_{\underline{n}}\left(  U\right)  U^{-\left\Vert \underline
{n}\right\Vert }T^{\underline{n}}%
\]%
\begin{align*}
&  +%
{\textstyle\sum\limits_{\emptyset\subsetneq J\subsetneq I_{0}}}
\text{ \ }%
{\textstyle\sum\limits_{\substack{\underline{m}\in B_{J}}}}
\left(  U-1\right)  U^{\left\Vert \underline{c}\right\Vert -\delta
-1}\mathcal{I}_{\underline{f_{J}}(\underline{m})}\left(  U\right)
U^{-\left\Vert \underline{c}\right\Vert -\left\Vert \underline{f_{J}%
}(\underline{m})\right\Vert }\frac{T^{\underline{f_{J}}(\underline{m})}}{%
{\textstyle\prod\limits_{i=1}^{r_{J}}}
\left(  1-U^{-1}T_{i}\right)  }\\
&  +\left(  U-1\right)  ^{d-1}U^{\delta-d+1}U^{-\left\Vert \underline
{c}\right\Vert }\frac{T^{\underline{c}}}{%
{\textstyle\prod\limits_{i=1}^{d}}
\left(  1-U^{-1}T_{i}\right)  },
\end{align*}
where $\underline{f_{J}}(\underline{m})=\left(  c_{1},\ldots,c_{r_{J}%
},m_{r_{J}+1},\ldots,m_{d}\right)  \in S$, with $m_{i}<c_{i}$, $r_{J}%
+1\leqslant i\leq d$, and $1\leqslant r_{J}<d$ are as in the explicit formula
given in Theorem \ref{teo1} (1), and $U$ is an indeterminate. We call
$\mathcal{Z}\left(  T_{1},\ldots,T_{d},U,S\right)  $ the universal zeta
function associated to $S$.
\end{definition}

By definition $\mathcal{Z}\left(  T_{1},\ldots,T_{d},U,S\right)  $ is
completely determined by $S$.

\begin{lemma}
\label{TheoremB} Assume that $k$ is big enough for $Y$. If $S=S\left(
\mathcal{O}\right)  $, then%
\[
Z\left(  T_{1},\ldots,T_{d},\mathcal{O}\right)  =\mathcal{Z}\left(
T_{1},\ldots,T_{d},U,S\right)  \mid_{U=\left[  \mathbb{A}_{k}^{1}\right]  }.
\]

\end{lemma}

\begin{proof}
The result follows from Corollary \ref{cor4}.
\end{proof}

\begin{remark}
Let $R$ be a ring. An \textit{additive invariant} is a map $\lambda
:Var_{k}\rightarrow R$ that satisfies the same conditions given in the
definition of the Grothendieck symbol in the category \ of $k$-algebraic
varieties (see e.g. \cite{L}, \cite{Ve}). By construction, the map
$Var_{k}\rightarrow K_{0}\left(  Var_{k}\right)  $ $:V\mapsto\left[  V\right]
$ is a universal additive invariant, i.e., the composition with $\left[
\cdot\right]  $ \ gives a bijection between the ring morphisms $K_{0}\left(
Var_{k}\right)  \rightarrow R$ and additive invariants $Var_{k}\rightarrow R$.

In the complex case, the Euler characteristic
\[
\chi\left(  X\right)  =%
{\textstyle\sum\nolimits_{i}}
\left(  -1\right)  ^{i}rank\left(  H_{c}^{i}\left(  X\left(  \mathbb{C}%
\right)  ,\mathbb{C}\right)  \right)
\]
gives rise to an additive invariant $\chi:Var_{\mathbb{C}}\rightarrow
\mathbb{Z}$. Since $\chi\left(  \mathbb{A}_{\mathbb{C}}^{1}\right)  =1$, the
Euler characteristic extends to a morphism $\mathcal{M}_{\mathbb{C}%
}\rightarrow\mathbb{Z}$. Then by specializing $\left[  \cdot\right]  $ to
$\chi\left(  \cdot\right)  $ in (\ref{zeta1}) and (\ref{zeta2}) we obtain two
`topological zeta functions', denoted by $\chi\left(  Z\left(  T_{1}%
,\ldots,T_{d},\mathcal{O}\right)  \right)  $ and $\chi\left(  Z\left(
T,\mathcal{O}\right)  \right)  $. From a computational point of view, these
specializations are obtained by replacing $\mathbb{L}$ by $\boldsymbol{1}$ in
the corresponding expressions.
\end{remark}

\begin{remark}
Let $(X,0)\subset(\mathbb{C}^{2},0)$ be a reduced plane curve singularity
defined by an equation $f=0$, with $f\in\mathcal{O}_{(\mathbb{C}^{2},0)}$
reduced. Let $h_{f}:V_{f}\rightarrow V_{f}$ be the monodromy transformation of
the singularity $f$ acting on its Milnor fiber $V_{f}$ (see \cite{A}). The
zeta function of $h_{f}$ (also called zeta function of the monodromy) is
defined to be
\[
\varsigma_{f}\left(  T\right)  :=\prod_{i\geq0}\left[  \mathrm{det}\left(
\mathrm{id}-T\cdot(h_{f})_{\ast}\mid_{H_{i}(V_{f};\mathbb{C})}\right)
\right]  ^{(-1)^{i+1}}.
\]

\end{remark}

The following theorem is due to Campillo, Delgado and Gusein-Zade
(\cite[Theorem 1]{CDG}):

\begin{theorem}
\label{TheoremMONO} [Campillo-Delgado-Gusein-Zade]\label{TheoremC} Put
$k=\mathbb{C}$. Then for any $\mathcal{O}=\mathcal{O}_{\left(  \mathbb{C}%
^{2},0\right)  }/\left(  f\right)  $, with $f\in\mathcal{O}_{\left(
\mathbb{C}^{2},0\right)  }$ reduced, and for any $S=S\left(  \mathcal{O}%
\right)  $, we have
\[
\varsigma_{f}\left(  T\right)  =\mathcal{Z}\left(  T_{1},\ldots,T_{d}%
,U,S\right)  \mid_{%
\begin{array}
[c]{l}%
T_{1}=\ldots=T_{d}=T\\
U=1
\end{array}
.}%
\]

\end{theorem}

\begin{proof}
As a consequence of the results of Campillo, Delgado, and Gusein-Zade (see
\cite{CDG1}, \cite{CDG2}, \cite{CDG4}) and Lemma \ref{rem5}, we have
$\chi\left(  Z\left(  T,\mathcal{O}\right)  \right)  =\varsigma_{f}\left(
T\right)  $, the zeta function of the monodromy $\varsigma_{f}\left(
T\right)  $ associated to the germ of function $f:\left(  \mathbb{C}%
^{2},0\right)  \rightarrow\left(  \mathbb{C},0\right)  $. By the previous
remark and Lemma \ref{TheoremB}, we have
\[
\chi\left(  Z\left(  T,\mathcal{O}\right)  \right)  =Z\left(  T,\mathcal{O}%
\right)  \mid_{\mathbb{L} \rightarrow1}=\mathcal{Z}\left(  T_{1},\ldots
,T_{d},U,S\right)  \mid_{%
\begin{array}
[c]{l}%
T_{1}=\ldots=T_{d}=T\\
U=1
\end{array}
.}%
\]

\end{proof}

\begin{remark}
In \cite{Z1} the second author introduced a Dirichlet series $Z(\mathrm{Ca}%
(Y),T)$ associated to the effective Cartier divisors on an algebraic curve
defined over a finite field $k=\mathbb{F}_{q}$. \ This zeta function admits an
Euler product of the form%
\[
Z(\mathrm{Ca}(Y),T)=%
{\textstyle\prod\limits_{P\in X}}
Z_{\mathrm{Ca}(Y)}(T,q,O_{P,Y}),
\]
with
\[
Z_{\mathrm{Ca}(Y)}(T,q,O_{P,Y}):=Z_{\mathrm{Ca}(Y)}(T,O_{P,Y})=%
{\textstyle\sum\limits_{I\subseteq O_{Y,P}}}
T^{\dim_{k}\left(  O_{P,Y}/I\right)  },
\]
where $I$ runs through all the principal ideals of $O_{P,Y}$. The notation
used here for the local factors of $Z(\mathrm{Ca}(Y),T)$ is a slightly
different to that used in \cite{Z1}. In addition, $Z_{\mathrm{Ca}%
(Y)}(T,O_{P,Y})=Z_{\mathrm{Ca}(Y)}(T,\widehat{O}_{P,Y})$, where $\widehat
{O}_{P,Y}$ is the completion of $O_{P,Y}$ with respect to the topology induced
by its maximal ideal. If $\widehat{O}_{P,Y}$ is totally rational, then
$Z_{\mathrm{Ca}(Y)}(T,\widehat{O}_{P,Y})$ is completely determined by the
semigroup of $\widehat{O}_{P,Y}$ (cf. \cite[Lemma 5.4 and Theorem 5.5]{Z1}).
\end{remark}

\begin{remark}
In the category of $\mathbb{F}_{q}$-algebraic varieties, $\left[
\cdot\right]  $ specializes to the counting rational points additive invariant
$\#\left(  \cdot\right)  $. In addition, for a cylindric subset $X\subset
\mathbb{P}\widetilde{\mathcal{O}}$ such that $X=\pi_{\underline{n}}^{-1}(Y)$
for a constructible subset $Y$ of $\mathbb{P}J_{\widetilde{\mathcal{O}}%
}^{\underline{n}}$, the only way to define the generalized Euler
characteristic $\chi_{g}(X)$ of $X$ is by specializing $\left[  \cdot\right]
$ to the counting map $\#\left(  \cdot\right)  $ that gives the number of
$\mathbb{F}_{q}$-rational points of a variety, i.e.,
\[
\chi_{g}(X)=\#(Y)\cdot q^{-||\underline{n}+\underline{1}||},
\]
see e.g. \cite{DM}. We denote by $\#\left(  Z\left(  T_{1},\ldots
,T_{d},\mathcal{O}\right)  \right)  $ the rational function obtained by
specializing $\left[  \cdot\right]  $ to $\#\left(  \cdot\right)  $. From a
computational point of view, $\#\left(  Z\left(  T_{1},\ldots,T_{d}%
,\mathcal{O}\right)  \right)  $ is obtained from $Z\left(  T_{1},\ldots
,T_{d},\mathcal{O}\right)  $ by replacing $\mathbb{L}$ by $q$.
\end{remark}

\begin{theorem}
\label{TheoremD}Let $\ k=\mathbb{F}_{q}$ and let $\mathcal{Z}\left(
T_{1},\ldots,T_{d},U,S\right)  $ be the universal zeta function for $S$. Let
$Y$ be an algebraic curve defined over $k$, and let $\widehat{O}_{P,Y}$ be the
completion of the local ring of $Y$ at a singular point $P$. Assume that $k$
is big enough for $Y$ and that $S=S\left(  \widehat{O}_{P,Y}\right)  $.

\noindent(1) For any $\mathcal{O}=\mathcal{O}_{\left(  \mathbb{C}%
^{2},0\right)  }/\left(  f\right)  $, with $f\in\mathcal{O}_{\left(
\mathbb{C}^{2},0\right)  }$ reduced, and $S=S\left(  \mathcal{O}\right)  $,%
\begin{align*}
Z_{\mathrm{Ca}(Y)}\left(  q^{-1}T,q,\widehat{O}_{P,Y}\right)   &  =\#\left(
Z\left(  T_{1},\ldots,T_{d},\widehat{O}_{P,Y}\right)  \right) \\
&  =\mathcal{Z}\left(  T_{1},\ldots,T_{d},U,S\right)  \mid_{%
\begin{array}
[c]{l}%
T_{1}=\ldots=T_{d}=T\\
U=q
\end{array}
.}%
\end{align*}
In particular $Z_{\mathrm{Ca}(Y)}\left(  q^{-1}T,q,\widehat{O}_{P,Y}\right)  $
depends only on $S$. In addition, and if $\widehat{O}_{P,Y}$\ is plane, then
$Z_{\mathrm{Ca}(Y)}\left(  q^{-1}T,q,\widehat{O}_{P,Y}\right)  $ is a complete
invariant of the equisingularity class of $\widehat{O}_{P,Y}$.

\noindent(2) For any $\mathcal{O}=\mathcal{O}_{\left(  \mathbb{C}%
^{2},0\right)  }/\left(  f\right)  $, with $f\in\mathcal{O}_{\left(
\mathbb{C}^{2},0\right)  }$, and $S=S\left(  \mathcal{O}\right)  $,
\[
Z_{\mathrm{Ca}(Y)}\left(  q^{-1}T,q,\widehat{O}_{P,Y}\right)  \mid_{q
\rightarrow1}=\varsigma_{f}\left(  T\right)  .
\]

\end{theorem}

\begin{proof}
1) Let $I=\left(  z_{1},\ldots,z_{d}\right)  \widehat{O}_{P,Y}\subseteq
\widehat{O}_{P,Y}$ be a principal ideal with
\[
\underline{n}=\left(  v_{1}\left(  z_{1}\right)  ,\ldots,v_{d}\left(
z_{d}\right)  \right)  .
\]
Since $\dim_{k}\left(  \widehat{O}_{P,Y}/I\right)  =\left\Vert \underline
{n}\right\Vert $, and the number of ideals with `codimension $\underline{n}$'
is finite -this number is denoted as $\#\left(  \mathcal{I}_{\underline{n}%
}\right)  $-, we have
\begin{equation}
Z_{\mathrm{Ca}(Y)}\left(  q^{-1}T,q,\widehat{O}_{P,Y}\right)  =\sum
_{\underline{n}\in S(\widehat{O}_{P,Y})}\#\left(  \mathcal{I}_{\underline{n}%
}\right)  q^{-||\underline{n}||}T^{||\underline{n}||}.\label{eqn:unaT}%
\end{equation}
On the other hand, by specializing $\left[  \cdot\right]  $ to $\#\left(
\cdot\right)  $ and by using the formula \ for $\left[  \mathcal{I}%
_{\underline{n}}\right]  $ given in Proposition \ref{lema8}, we obtain the
explicit formula given for $\#\left(  \mathcal{I}_{\underline{n}}\right)  $ in
\cite[Lemma 5.4 ]{Z1}, hence%
\begin{align*}
Z_{\mathrm{Ca}(Y)}\left(  q^{-1}T,q,\widehat{O}_{P,Y}\right)   &  =\#\left(
Z\left(  T_{1},\ldots,T_{d},\widehat{O}_{P,Y}\right)  \right)  \\
&  =Z\left(  T,\ldots,T,\widehat{O}_{P,Y}\right)  \mid_{\mathbb{L}\rightarrow
q}\\
&  =\mathcal{Z}\left(  T_{1},\ldots,T_{d},U,S\right)  \mid_{%
\begin{array}
[c]{l}%
T_{1}=\ldots=T_{d}=T\\
U=q
\end{array}
,}%
\end{align*}
where in the last equality we used Lemma \ref{TheoremB}.

2) From the first part and by using Theorem \ref{TheoremC}, we have%
\begin{align*}
Z_{Ca(Y)}\left(  q^{-1}T,q,\widehat{O}_{P,Y}\right)   &  \mid_{q\rightarrow
1}=\mathcal{Z}\left(  T_{1},\ldots,T_{d},U,S\right)  \mid_{%
\begin{array}
[c]{l}%
T_{1}=\ldots=T_{d}=T\\
U=1
\end{array}
}\\
&  =\varsigma_{f}\left(  T\right)  .
\end{align*}

\end{proof}

\section{\label{functionalequations}Functional Equations}

In this section $k$ is a field of characteristic $p\geq0$, and $\mathcal{O}$
is a Gorenstein and totally rational ring. Let $S=S(\mathcal{O})$. We give
functional equations for $Z(T_{1}, \ldots, T_{d}, \mathcal{O})$,
$\mathcal{Z}(T_{1},\ldots,T_{d},U,S)$ and for other Poincar\'e series.
\medskip

Recall that for any $\underline{n}\in\mathbb{Z}^{d}$, we have $l(\underline
{n})=\dim_{k}\left(  \mathcal{O}/J_{\underline{n}}(\mathcal{O})\right)  $,
with $J_{\underline{n}}=\{\underline{z}\in\mathcal{O}\mid\underline
{v}(\underline{z})\geq\underline{n}\}$ (cf. Section \ref{section:zetafunction}%
). In addition we have:
\begin{equation}
l(\underline{n})=l(\underline{n}-\underline{e}_{i})+\dim_{k}\left(
J_{\underline{n}-\underline{e}_{i}}(\mathcal{O})/J_{\underline{n}}%
(\mathcal{O})\right)  ~~~~~~\text{for~all}\mathrm{~}\underline{n}\in
\mathbb{Z}^{d}. \label{eqn:elece}%
\end{equation}

The following result can be found in \cite[Theorem (3.6)]{CDK}:

\begin{lemma}
[Campillo-Delgado-Kiyek]\label{lemma:kiyek} For any $\underline{n}
\in\mathbb{Z}^{d}$ and any $i \in\{1, \ldots,d \}$ we have
\[
\dim_{k} \left(  J_{\underline{n}}(\mathcal{O})/J_{\underline{n}+\underline
{e}_{i}}(\mathcal{O})\right)  + \dim_{k} \left(  J_{\underline{c} -
\underline{n} - \underline{e}_{i}}(\mathcal{O})/J_{\underline{c}-\underline
{n}}(\mathcal{O}) \right)  =1.
\]

\end{lemma}

The following result will be used in the proof of the functional equation:

\begin{lemma}
\label{eles}
\begin{equation}
l(\underline{c}-\underline{n})-l(\underline{n})=\delta-||\underline{n}||,~
~\underline{n}\in\mathbb{Z}^{d}.\nonumber
\end{equation}

\end{lemma}

\begin{proof}
We use induction on $||\underline{m}||:=\sum_{i=1}^{d}|m_{i}|$, where
$\underline{m}=(m_{1},\ldots,m_{d})\in\mathbb{Z}^{d}$. For $||\underline
{m}||=0$ we have $\underline{m}=\underline{0}$. In this case $l(\underline
{0})=0$ and $l(\underline{c})=\delta$, and the result is true. Assume, as
induction hypothesis, that the result is true for every $\underline{m}%
\in\mathbb{Z}^{d}$ with $||\underline{m}||\leq k$ for some $k\geq1$. From the
induction hypothesis, we have the following two formulas: (i) if
$0<||\underline{m}||\leq k$ and $m_{i}\geq1$ for some $i\in\{1,\ldots,d\}$,
then for $\underline{m}-\underline{e}_{i}$,
\begin{equation}
l(\underline{c}-(\underline{m}-\underline{e}_{i}))-l(\underline{m}%
-\underline{e}_{i})=\delta-||\underline{m}-\underline{e}_{i}||.
\label{hipotesis}%
\end{equation}
(ii) If $\ 0<||\underline{m}||\leq k$ and $m_{i}\leq0$, then for some
$i\in\{1,\ldots,d\}$, $m_{i}<0$. Then for $\underline{m}+\underline{e}_{i}$,
\begin{equation}
l(\underline{c}-(\underline{m}+\underline{e}_{i}))-l(\underline{m}%
+\underline{e}_{i})=\delta-||\underline{m}+\underline{e}_{i}||.
\label{hipotesis2}%
\end{equation}

We now verify the validity of the result for $||\underline{m}||=k+1$. If
$m_{i}\geq1$ for some $i\in\{1,\ldots,d\}$, by applying (\ref{eqn:elece})
\begin{equation}
l(\underline{c}-\underline{m})-l(\underline{m})=l(\underline{c}-\underline
{m})-l(\underline{m}-\underline{e}_{i})-\dim_{k}\left(  J_{\underline
{m}-\underline{e}_{i}}(\mathcal{O})/J_{\underline{m}}(\mathcal{O})\right)
,\nonumber
\end{equation}
we now use Lemma \ref{lemma:kiyek} and (\ref{eqn:elece}) to get
\begin{align}
l(\underline{c}-\underline{m})-l(\underline{m})  &  =l(\underline
{c}-\underline{m})-l(\underline{m}-\underline{e}_{i})-\left(  1-\dim
_{k}\left(  J_{\underline{c}-\underline{m}}(\mathcal{O})/J_{\underline
{c}-\underline{m}+\underline{e}_{i}}(\mathcal{O})\right)  \right) \nonumber\\
&  =l(\underline{c}-\underline{m})+\dim_{k}\left(  J_{\underline{c}%
-\underline{m}}(\mathcal{O})/J_{\underline{c}-\underline{m}+\underline{e}_{i}%
}(\mathcal{O})\right)  -l(\underline{m}-\underline{e}_{i})-1\nonumber\\
&  =l(\underline{c}-(\underline{m}-\underline{e}_{i}))-l(\underline
{m}-\underline{e}_{i})-1.\nonumber
\end{align}

Finally, by applying induction hypothesis (\ref{hipotesis}) we get
\begin{equation}
l(\underline{c}-\underline{m})-l(\underline{m})=\delta-||\underline
{m}||.\nonumber
\end{equation}

In the case in which $m_{i}<0$, for some $i\in\{1,\ldots,d\}$, we apply the
previous reasoning and induction hypothesis (\ref{hipotesis2}) to get
\[
l(\underline{c}-\underline{m})-l(\underline{m})=\delta-||\underline{m}||.
\]

\end{proof}

\begin{remark}
We note that $\mathcal{I}_{\underline{n}}=\emptyset$ whenever $\underline
{n}\notin S$, thus, $\left[  \mathcal{I}_{\underline{n}}\right]  =0$ if
$\underline{n}\notin S$. We can write $Z(T_{1},\ldots,T_{d},\mathcal{O})$ as
follows:
\[
Z\left(  T_{1},\ldots,T_{d},\mathcal{O}\right)  ={\textstyle\sum
\nolimits_{\underline{n}\in\mathbb{Z}^{d}}}\left[  \mathcal{I}_{\underline{n}%
}\right]  \mathbb{L}^{-\left\Vert \underline{n}\right\Vert }T^{\underline{n}%
}.
\]

\end{remark}

\begin{theorem}
\label{TheoremE} Let $\mathcal{O}$ be a Gorenstein and totally rational ring.
Assume that $\mathcal{J}\cong\left(  G_{m}\right)  ^{d-1}\times\left(
G_{a}\right)  ^{\delta-d+1}$, then
\[
\noindent(1)\text{ }Z(\mathbb{L}T_{1},\ldots,\mathbb{L}T_{d},\mathcal{O}%
)=\mathbb{L}^{\delta-d}\cdot T^{\underline{c}-\underline{1}}\cdot\frac
{\prod_{i=1}^{d}(1-\mathbb{L}T_{i})}{\prod_{i=1}^{d}(T_{i}-1)}\cdot
Z(T_{1}^{-1},\ldots,T_{d}^{-1},\mathcal{O});
\]%
\[
\noindent(2)\text{ }\mathcal{Z}(UT_{1},\ldots,UT_{d},U,S)=U^{\delta-d}\cdot
T^{\underline{c}-\underline{1}}\cdot\frac{\prod_{i=1}^{d}(1-UT_{i})}%
{\prod_{i=1}^{d}(T_{i}-1)}\cdot\mathcal{Z}(T_{1}^{-1},\ldots,T_{d}^{-1},U,S).
\]

\end{theorem}

\begin{proof}
(1)~We first note that
\begin{align*}
\left(  \prod_{i=1}^{d}(T_{i}-1)\right)  Z\left(  \mathbb{L}T_{1}%
,\ldots,\mathbb{L}T_{d},\mathcal{O}\right)   &  =\left(  \prod_{i=1}^{d}%
(T_{i}-1)\right)  \sum_{\underline{n}\in\mathbb{Z}^{d}}\left[  \mathcal{I}%
_{\underline{n}}\right]  T^{\underline{n}}\\
&  =\sum_{\underline{n}\in\mathbb{Z}^{d}}\sum_{J\subseteq I_{0}}%
(-1)^{d-\#J}\left[  \mathcal{I}_{\underline{n}}\right]  T^{\underline
{n}+\underline{1}_{J}}\\
&  =\sum_{\underline{n}\in\mathbb{Z}^{d}}\sum_{J\subseteq I_{0}}%
(-1)^{d-\#J}\left[  \mathcal{I}_{\underline{n}-\underline{1}_{J}}\right]
T^{\underline{n}},
\end{align*}
where $I_{0}=\left\{  1,2,\ldots,d\right\}  $ and for $J\subseteq I_{0}$,
$\underline{1}_{J}$ is the element of $\mathbb{N}^{d}$ whose $i$-th component
is equal to $1$ or $0$, accordingly if $i\in J$, or if $i\notin J$,
respectively. If $\underline{n}-\underline{1}_{J}\notin S$, then $\left[
\mathcal{I}_{\underline{n}-\underline{1}_{J}}\right]  =0$; if $\underline
{n}-\underline{1}_{J}\in S$, then by applying Proposition \ref{lema8},
\begin{align*}
&  \left(  \prod_{i=1}^{d}(T_{i}-1)\right)  Z\left(  \mathbb{L}T_{1}%
,\ldots,\mathbb{L}T_{d},\mathcal{O}\right) \\
&  =\frac{\mathbb{L}}{\mathbb{L}-1}\sum_{\underline{n}\in\mathbb{Z}^{d}}%
\sum_{J\subseteq I_{0}}(-1)^{d-\#J}\sum_{I\subseteq I_{0}}(-1)^{\#I}%
\mathbb{L}^{||\underline{n}-\underline{1}_{J}||-l(\underline{n}+\underline
{1}_{I}-\underline{1}_{J})}T^{\underline{n}}.
\end{align*}
Taking into account that $\mathcal{O}$ is Gorenstein, i.e. $||\underline
{c}||=2\delta$, and applying Lemma \ref{eles},
\[
l(\underline{n}+\underline{1}_{I}-\underline{1}_{J})=||\underline
{n}+\underline{1}_{I}-\underline{1}_{J}||+l(\underline{c}-\underline
{n}-\underline{1}_{I}+\underline{1}_{J})-\delta,
\]
and $\left(  \prod_{i=1}^{d}(T_{i}-1)\right)  Z\left(  \mathbb{L}T_{1}%
,\ldots,\mathbb{L}T_{d},\mathcal{O}\right)  $ becomes%
\begin{align*}
&  \frac{\mathbb{L}}{\mathbb{L}-1}\sum_{\underline{n}\in\mathbb{Z}^{d}}%
\sum_{J\subseteq I_{0}}(-1)^{d-\#J}\sum_{I\subseteq I_{0}}(-1)^{\#I}%
\mathbb{L}^{||\underline{n}-\underline{1}_{J}||-||\underline{n}+\underline
{1}_{I}-\underline{1}_{J}||-l(\underline{c}-\underline{n}-\underline{1}%
_{I}+\underline{1}_{J})+\delta}T^{\underline{n}}\\
&  =\frac{\mathbb{L}}{\mathbb{L}-1}\sum_{\underline{n}\in\mathbb{Z}^{d}}%
\sum_{I\subseteq I_{0}}\sum_{J\subseteq I_{0}}(-1)^{d-\#J+\#I}\mathbb{L}%
^{-\delta+||\underline{n}||}\mathbb{L}^{||\underline{c}-\underline
{n}-\underline{1}_{I}||-l(\underline{c}-\underline{n}-\underline{1}%
_{I}+\underline{1}_{J})}T^{\underline{n}}\\
&  =\sum_{\underline{n}\in\mathbb{Z}^{d}}\sum_{I\subseteq I_{0}}%
(-1)^{d+\#I}\mathbb{L}^{-\delta+||\underline{n}||}\left(  \frac{\mathbb{L}%
}{\mathbb{L}-1}\sum_{J\subseteq I_{0}}(-1)^{\#J}\mathbb{L}^{||\underline
{c}-\underline{n}-\underline{1}_{I}||-l(\underline{c}-\underline{n}%
-\underline{1}_{I}+\underline{1}_{J})}\right)  T^{\underline{n}}\\
&  =\sum_{\underline{n}\in\mathbb{Z}^{d}}\mathbb{L}^{-\delta+||\underline
{n}||}\sum_{I\subseteq I_{0}}(-1)^{d-\#I}\left[  \mathcal{I}_{\underline
{c}-\underline{n}-\underline{1}_{I}}\right]  T^{\underline{n}}%
\end{align*}

\begin{align*}
&  =\sum_{\underline{n}\in\mathbb{Z}^{d}}\sum_{I\subseteq I_{0}}%
(-1)^{d-\#I}\left[  \mathcal{I}_{\underline{c}-\underline{n}-\underline{1}%
_{I}}\right]  \mathbb{L}^{-\delta+||\underline{n}||}T^{\underline{n}}\\
&  =\sum_{\underline{n}\in\mathbb{Z}^{d}}\sum_{I\subseteq I_{0}}%
(-1)^{d-\#I}\left[  \mathcal{I}_{\underline{c}-\underline{n}-\underline{1}%
_{I}}\right]  \mathbb{L}^{\delta-||\underline{c}-\underline{n}||}%
T^{\underline{n}}\\
&  =\sum_{\underline{m}\in\mathbb{Z}^{d}}\sum_{I\subseteq I_{0}}%
(-1)^{d-\#I}\left[  \mathcal{I}_{\underline{m}-\underline{1}_{I}}\right]
\mathbb{L}^{\delta-||\underline{m}||}T^{\underline{c}-\underline{m}}\\
&  =\mathbb{L}^{\delta}T^{\underline{c}}\sum_{\underline{m}\in\mathbb{Z}^{d}%
}\sum_{I\subseteq I_{0}}(-1)^{d-\#I}\left[  \mathcal{I}_{\underline
{m}-\underline{1}_{I}}\right]  (\mathbb{L}T_{1})^{-m_{1}}\cdot\ldots
\cdot(\mathbb{L}T_{d})^{-m_{d}}\\
&  =\mathbb{L}^{\delta}T^{\underline{c}}\left(  \prod_{i=1}^{d}\left(  \left(
\mathbb{L}T_{i}\right)  ^{-1}-1\right)  \right)  \sum_{\underline{m}%
\in\mathbb{Z}^{d}}\left[  \mathcal{I}_{\underline{m}}\right]  (\mathbb{L}%
T_{1})^{-m_{1}}\cdot\ldots\cdot(\mathbb{L}T_{d})^{-m_{d}}\\
&  =\mathbb{L}^{\delta-d}T^{\underline{c}-\underline{1}}\left(  \prod
_{i=1}^{d}\left(  1-\mathbb{L}T_{i}\right)  \right)  Z\left(  T_{1}%
^{-1},\ldots,T_{d}^{-1},\mathcal{O}\right)  .
\end{align*}
(2)~ The functional equation for $\mathcal{Z}(T_{1},\ldots,T_{d},U,S)$ follows
from the first part by Definition \ref{universalzetafunction} and Theorem
\ref{teo1}.
\end{proof}

It is worth mentioning that, since we have \emph{not} shown that
\[
{\sum\nolimits_{\underline{n}\in\mathbb{Z}^{d}}}\mathcal{I}_{\underline{n}%
}\left(  U\right)  U^{-\left\Vert \underline{n}\right\Vert }T^{\underline{n}%
}=\mathcal{Z}(T_{1},\ldots,T_{d},U,S),
\]
it is necessary to show first the functional equation for $Z\left(
T_{1},\ldots,T_{d},\mathcal{O}\right)  $.

\begin{corollary}
If $%
{\textstyle\prod\nolimits_{i=1}^{d}}
\left(  1-\mathbb{L}^{-1}T_{i}\right)  Z\left(  T_{1},\ldots,T_{d}%
,\mathcal{O}\right)  =M\left(  T_{1},\ldots,T_{d},\mathcal{O}\right)  $, with
\[
M\left(  T_{1},\ldots,T_{d},\mathcal{O}\right)  =\sum_{\underline{0}%
\leq\underline{i}\leq\underline{c}}a\underline{_{i}}T^{\underline{i}},
\]
then (1) $M\left(  \mathbb{L}T_{1},\ldots,\mathbb{L}T_{d},\mathcal{O}\right)
=\mathbb{L}^{\delta}T^{\underline{c}}M\left(  T_{1}^{-1},\ldots,T_{d}%
^{-1},\mathcal{O}\right)  .$ (2) $a\underline{_{i}}=a_{\underline{c}%
-}\underline{_{i}}\mathbb{L}^{\delta-\left\Vert \underline{i}\right\Vert }$,
for $\underline{0}\leq\underline{i}\leq\underline{c}$. In particular,
$a_{\underline{c}}=\mathbb{L}^{-\delta}$, since $a_{\underline{0}}=1$, and
then the degree of $M\left(  T_{1},\ldots,T_{d},\mathcal{O}\right)  $\ is
$\left\Vert \underline{c}\right\Vert $.
\end{corollary}

\begin{remark}
(1) By specialization several functional equations can be obtained, among
them,
\[
Z(\mathbb{L}T,\mathcal{O})=\mathbb{L}^{\delta-d}\cdot T^{||\underline
{c}-\underline{1}||}\cdot\left(  \frac{1-\mathbb{L}T}{T-1}\right)  ^{d}\cdot
Z(T^{-1},\mathcal{O}).
\]

\noindent(2) By Lemma \ref{rem5} one also obtains the functional equations for
the generalized Poincar\'{e} series (see also \cite[Theorem 5.4.3]{M}):%
\[
P_{g}\left(  \mathbb{L}T_{1},\ldots,\mathbb{L}T_{d}, \mathcal{O}\right)
=\mathbb{L}^{\delta-d}\cdot T^{\underline{c}-\underline{1}}\cdot\frac
{\prod_{i=1}^{d}(1-\mathbb{L}T_{i})}{\prod_{i=1}^{d}(T_{i}-1)}\cdot
P_{g}\left(  T_{1}^{-1},\ldots,T_{d}^{-1}, \mathcal{O}\right)  .
\]

\noindent(3) Let $\mathcal{O}=\mathcal{O}_{\left(  \mathbb{C}^{2},0\right)
}/\left(  f\right)  $, where $f\in\mathcal{O}_{\left(  \mathbb{C}%
^{2},0\right)  }$ is reduced. Then
\[
\varsigma_{f}\left(  T\right)  =(-1)^{d}T^{||\underline{c}-\underline{1}%
||}\cdot\varsigma_{f}\left(  T^{-1}\right)  .
\]

\end{remark}

\section{\label{ejemplos}Examples}

\subsection{Example\label{ej1}}

Set $\mathcal{O}=\mathbb{C} \{\!\{ x,y \}\!\} /\left(  x^{3}-y^{2}\right)  $
and $\widetilde{\mathcal{O}}=\mathbb{C} \{\!\{ t \}\!\} $, then%

\[
\mathcal{O}=\left\{
{\textstyle\sum\nolimits_{i=0}^{\infty}}
a_{i}t^{i}\in\widetilde{\mathcal{O}}\mid a_{1}=0\right\}  =\mathbb{C}%
+t^{2}\mathbb{C} \{\!\{ t \} \!\} ,
\]
and
\[
\mathcal{O}^{\times}=\left\{
{\textstyle\sum\nolimits_{i=0}^{\infty}}
a_{i}t^{i}\in\widetilde{\mathcal{O}}\mid a_{0}\neq0,a_{1}=0\right\}  .
\]
The semigroup of values is the set $\left\{  0\right\}  \cup\left\{
n\in\mathbb{N}\mid n\geqslant2\right\}  $, and $\left[  \pi_{\underline
{c}-\underline{1}}\left(  \mathcal{O}^{\times}\right)  \right]  =\left[
\mathbb{C}^{\times}\times\left\{  \text{point}\right\}  \right]  $
$=\mathbb{L}-1$. The group $\mathcal{J}$ is isomorphic to $\left\{  1+bt\mid
b\in\mathbb{C}\right\}  $, where the product is defined as $\left(
1+b_{0}t\right)  \left(  1+b_{1}t\right)  =1+\left(  b_{0}+b_{1}\right)  t$,
and the identity is $1$. We now compute the zeta function of $Z\left(
T,\mathcal{O}\right)  $. We first note that $\left[  \mathcal{I}_{0}\right]
=\left[  \left\{  \text{point}\right\}  \right]  =1$. To compute $\left[
\mathcal{I}_{k}\right]  $ for $k\geqslant2$, we fix a set of polynomial
representatives $\left\{  \mu\right\}  $ of $\mathcal{J}$ in $\widetilde
{\mathcal{O}}$. If $I=z\mathcal{O}$, $v(z)=k$, then $z=t^{k}\left(
1+bt\right)  v$, with $1+bt\in\mathcal{J}$ and $v\in\mathcal{O}^{\times}$. The
correspondence $z\rightarrow1+bt$ gives a bijection between $\mathcal{I}_{k}$
and $\mathcal{J}$, for $k\geqslant2$, therefore $\left[  \mathcal{I}%
_{k}\right]  =\mathbb{L}$, for $k\geqslant2$, and $Z\left(  T,\mathcal{O}%
\right)  =\frac{1-\mathbb{L}^{-1}T+\mathbb{L}^{-1}T^{2}}{1-\mathbb{L}^{-1}T}$.
Note that each $I$ in $\mathcal{I}_{k}$\ corresponds to a $\mu\in\mathcal{J}$
such that $t^{k}\mu\in\mathcal{O}$.

By specializing $\left[  \cdot\right]  $ to the Euler characteristic
$\chi\left(  \cdot\right)  $, we obtain $\chi\left(  Z\left(  T,\mathcal{O}%
\right)  \right)  $ $=\frac{1-T+T^{2}}{1-T}=\varsigma_{f}\left(  T\right)  $.
By applying a theorem of A'Campo (see \cite{A}) it is possible to verify that
$\chi\left(  Z\left(  T,\mathcal{O}\right)  \right)  $ is the zeta function of
the monodromy at the origin of the mapping $f:\mathbb{C}^{2}\rightarrow
\mathbb{C}$, where $\ f\left(  x,y\right)  =x^{3}-y^{2}$.

Let $\mathbb{F}_{q}$ be a finite field with $q$ elements. Let us consider the
local ring $\mathcal{A}= \mathbb{F}_{q} [\![ x,y ]\!] /\left(  x^{3}%
-y^{2}\right)  $ which is totally rational over $\mathbb{F}_{q}$. Observe that
$\delta=1$ and $\mathcal{J}\cong(\mathbb{F}_{q},+,0)$. By specializing
$\left[  \cdot\right]  $ to $\#\left(  \cdot\right)  $, we get $\#\left(
Z\left(  T,\mathcal{O}\right)  \right)  =Z\left(  q^{-1}T,\mathcal{A}\right)
$, where $Z\left(  T,\mathcal{A}\right)  =\frac{1-T+qT^{2}}{1-T}$ is the local
factor of the zeta function $Z\left(  \text{\textrm{Ca}}\left(  X\right)
,T\right)  $ at the origin, here $X$ is the projective curve over
$\mathbb{F}_{q}$ defined by $f\left(  x,y\right)  =x^{3}-y^{2}\in
\mathbb{F}_{q}\left[  x,y\right]  $. Note that $\lim_{q\rightarrow1}Z\left(
T,\mathcal{A}\right)  =\varsigma_{f}\left(  T\right)  $, see \cite[Example
5.6]{Z1}.

\subsection{Example\label{ej2}}

Set $\mathcal{O}=\mathbb{C}\{\!\{x,y\}\!\}/\left(  y^{2}-x^{4}+x^{5}\right)  $
and $\widetilde{\mathcal{O}}=\mathbb{C}\{\!\{t_{1}\}\!\}\times\mathbb{C}%
\{\!\{t_{2}\}\!\}$, then
\[
\mathcal{O}=\left\{  \left(
{\textstyle\sum\nolimits_{i=0}^{\infty}}
a_{i,1}t_{1}^{i},%
{\textstyle\sum\nolimits_{i=0}^{\infty}}
a_{i,2}t_{2}^{i}\right)  \in\widetilde{\mathcal{O}}\mid a_{0,1}=a_{0,2}\text{,
}a_{1,1}=a_{1,2}\right\}  .
\]
\ The conductor ideal is $\mathcal{F}=\left(  t_{1}^{2},t_{2}^{2}\right)
\widetilde{\mathcal{O}}$, and the semigroup $S$ is equal to%
\[
\left\{  \left(  0,0\right)  \right\}  \cup\left\{  \left(  1,1\right)
\right\}  \cup\left\{  \left(  k_{1},k_{2}\right)  \in\mathbb{N}^{2}\mid
k_{1}\geqslant2\text{, }k_{2}\geqslant2\right\}  .
\]
Note that $\left[  \pi_{\underline{c}-\underline{1}}\left(  \mathcal{O}%
^{\times}\right)  \right]  =\left(  \mathbb{L}-1\right)  \mathbb{L}$. The
group $\mathcal{J}$ is isomorphic to
\[
\left\{  \left(  a+bt_{1},1\right)  \mid a\in\mathbb{C}^{\times},\text{ }%
b\in\mathbb{C}\right\}  ,
\]
where the product is defined as
\[
\left(  a_{0}+b_{0}t_{1},1\right)  \left(  a_{1}+b_{1}t_{1},1\right)  =\left(
a_{0}a_{1}+\left(  a_{0}b_{1}+a_{1}b_{0}\right)  t_{1},1\right)  .
\]
Therefore $\left[  \mathcal{J}\right]  =\left(  \mathbb{L}-1\right)
\mathbb{L}$, and $\left[  \mathcal{I}_{\underline{n}}\right]  =\left[
\mathcal{J}\right]  $, for $\underline{n}\geqslant(2,2)$. To compute $\left[
\mathcal{I}_{\left(  1,1\right)  }\right]  $\ we use the fact that each $I$ in
$\mathcal{I}_{\left(  1,1\right)  }$\ corresponds to a point of $\underline
{\mu}=\left(  a+bt_{1},1\right)  \in\mathcal{J}$ such that $\left(
t_{1,}t_{2}\right)  \underline{\mu}\in\mathcal{O}$, thus we have to determine
all the $a\in\mathbb{C}^{\times}$ and $b\in\mathbb{C}$ such that
\[
\left(  t_{1},t_{2}\right)  \left(  a+bt_{1},1\right)  =\left(  at_{1}%
+bt_{1}^{2},t_{2}\right)  \in\mathcal{O}%
\]
(here the product is in $\widetilde{\mathcal{O}}$ and not in $\mathcal{J}$),
then $a=1$, $b\in\mathbb{C}$ and thus $\left[  \mathcal{I}_{\left(
1,1\right)  }\right]  =\mathbb{L}$, and $Z\left(  T_{1},T_{2},\mathcal{O}%
\right)  $ is equal to
\[
\frac{1-\mathbb{L}^{-1}T_{1}-\mathbb{L}^{-1}T_{2}+\left(  \mathbb{L}%
^{-1}+\mathbb{L}^{-2}\right)  T_{1}T_{2}-\mathbb{L}^{-2}T_{1}T_{2}%
^{2}-\mathbb{L}^{-2}T_{1}^{2}T_{2}+\mathbb{L}^{-2}T_{1}^{2}T_{2}^{2}}{\left(
1-\mathbb{L}^{-1}T_{1}\right)  \left(  1-\mathbb{L}^{-1}T_{2}\right)  }.
\]

By specializing \ $\left[  \cdot\right]  $ to $\chi\left(  \cdot\right)  $ we
have $\chi\left(  Z\left(  T,\mathcal{O}\right)  \right)  =1+T^{2}%
=\varsigma_{f}\left(  T\right)  $, that are the Alexander polynomial and the
zeta function of the monodromy of the germ of mapping $f:\mathbb{C}%
^{2}\rightarrow\mathbb{C}:$ $\left(  x,y\right)  \mapsto y^{2}-x^{4}+x^{5}$ at
the origin.

Set $\mathcal{A}=\mathbb{F}_{q} [\![ x,y ]\!] /\left(  y^{2}-x^{4}%
+x^{5}\right)  $. Observe that $\delta=2$ and $\mathcal{J}\cong((\mathbb{F}%
_{q})^{\times},\cdot)\times(\mathbb{F}_{q},+,0)$. By specializing $\left[
\cdot\right]  $ to $\#\left(  \cdot\right)  $, we obtain the equality
$\#\left(  Z\left(  T,\mathcal{O}\right)  \right)  =Z\left(  q^{-1}%
T,\mathcal{A}\right)  $, where $Z\left(  T,\mathcal{A}\right)  =\frac
{1-2T+(q+1)T^{2}-2qT^{3}+q^{2}T^{4}}{\left(  1-T\right)  ^{2}}$ is the local
factor of the zeta function $Z\left(  \mathrm{Ca}\left(  X\right)  ,T\right)
$ at the origin, here $X$ the projective curve over $\mathbb{F}_{q}$ defined
by $f\left(  x,y\right)  =y^{2}-x^{4}+x^{5}\in\mathbb{F}_{q}\left[
x,y\right]  $. Note that $\lim_{q\rightarrow1}Z\left(  T,\mathcal{A}\right)
=\varsigma_{f}\left(  T\right)  $.

\subsection{Example\label{ej3}}

Set $\mathcal{O}\mathbb{=C} \{\!\{ t^{3},t^{4},t^{5} \}\!\} $ and
$\widetilde{\mathcal{O}}=\mathbb{C} \{\!\{ t_{1} \}\!\} $. The embedding
dimension of $\mathcal{O}$\ is three. The group $\mathcal{J}$ is isomorphic to
$\left\{  1+at+bt^{2}\mid a,b\in\mathbb{C}\right\}  $, where the product is
defined as
\[
\left(  1+a_{0}t+b_{0}t^{2}\right)  \left(  1+a_{1}t+b_{1}t^{2}\right)
=1+\left(  a_{0}+a_{1}\right)  t+\left(  b_{0}+b_{1}+a_{0}a_{1}\right)
t^{2}.
\]
The zeta function of this ring is $Z\left(  T,\mathcal{O}\right)
=\frac{1-\mathbb{L}^{-1}T+\mathbb{L}^{-1}T^{3}}{1-\mathbb{L}^{-1}T}$, and
$\chi\left(  Z\left(  T,\mathcal{O}\right)  \right)  =\frac{1-T+T^{3}}{1-T}$.
This rational function should be `the monodromy zeta function of $\mathcal{O}%
$,' but this cannot be explained from the point of view of A'Campo paper
\cite{A} . It seems that the connection between $\chi\left(  Z\left(
T,\mathcal{O}\right)  \right)  $ and the \textquotedblleft topology of
$\mathcal{O}$\textquotedblright\ is not completely understood.

\end{document}